\documentclass[preprint,3p]{elsarticle1}

\usepackage{graphicx}

\usepackage{amssymb}

\usepackage{lineno}

\usepackage{amsmath}
\usepackage{amsthm}

\usepackage{mathrsfs}

\usepackage{psfrag}
\usepackage{color}

\journal{}

\newtheorem{theorem}{Theorem}
\newtheorem{lemma}[theorem]{Lemma}
\newtheorem{example}[theorem]{Example}
\newtheorem{definition}[theorem]{Definition}

\newtheorem{remark}{Remark}

\newtheoremstyle{algstyle}%
  {10mm}       
  {10mm}       
  {\tt}   
  {0pt}        
  {\bfseries}  
  {\newline}   
  {10mm}       
  {\thmname{#1}\thmnumber{ #2}\thmnote{ (#3)}}          
\theoremstyle{algstyle}

\newtheoremstyle{algdashstyle}%
  {10mm}       
  {10mm}       
  {\tt}   
  {0pt}        
  {\bfseries}  
  {\newline}   
  {10mm}       
  {\thmname{#1}\thmnumber{ #2}$'$\thmnote{ (#3)}}          
\theoremstyle{algdashstyle}

\newcommand{\bbmatrix}[1]{%
\begin{bmatrix} #1 \end{bmatrix}%
}

\newcommand{\ppmatrix}[1]{%
\begin{pmatrix} #1 \end{pmatrix}%
}

\newcommand{\V}{\mbox{$\cal V$}} 
\newcommand{\F}{\mbox{$\cal F$}} 

    \newcommand{\0}{{\mathbf 0}}        

\newcommand{\A}[0]{{\cal A}}      						
\newcommand{\C}[0]{{\cal C}}      						
\newcommand{\T}[0]{{\cal T}}      						
\newcommand{\D}[0]{{\cal D}}                    	
 
\newcommand{\e}{\mbox{$\bf e$}} 
\newcommand{\G}[0]{{\cal G}}                       
  
\newcommand{\K}[0]{{\cal K}}                       

\newcommand{\M}{\mbox{$\cal M$}}

\newcommand{\N}[0]{{\cal N}}    							

\newcommand{\B}{\mbox{${\cal B}$}}  				
\newcommand{\Reg}{\mbox{${\cal R}$}} 				

\newcommand{\U}{\mbox{${\bf u}$}}             		
\newcommand{\x}{\mbox{${\bf x}$}}             		

\begin{document}

\begin{frontmatter}



\title{Higher dimensional electrical circuits and the matroid dual of a nonplanar graph}

\author[hari]{Hariharan Narayanan}
\ead{hariharan.narayanan@tifr.res.in}
\author[hn]{H. Narayanan \corref{cor2}}
\ead{hn@ee.iitb.ac.in}
\cortext[cor2]{Corresponding author}
\address[hari]{School of Technology and Computer Science, Tata Institute of Fundamental Research}
\address[hn]{Department of Electrical Engineering, Indian Institute of Technology Bombay}

\begin{abstract}
In this paper we describe a physical problem,  based on electromagnetic fields, 
whose topological constraints are higher dimensional versions of Kirchhoff's laws, involving
$2-$ simplicial complexes embedded in $\mathbb{R} ^3$ rather than graphs. However, we show that, for the   skeleton of this complex, involving only triangles  and edges, we can build a matroid dual
 which is a graph. On this graph we build an `ordinary'  electrical 
circuit, solving which we obtain  the solution to our original problem.
Construction of this graph  is through a `sliding' algorithm which simulates sliding on the surfaces 
of the  triangles, moving from one triangle to another which shares an edge with it
but which also is adjacent with respect to the embedding  of the complex in $\mathbb{R} ^3.$
For this purpose, the only information needed is the order in which we encounter
the triangles incident at an edge, when we rotate say clockwise with respect to the 
orientation of the edge.
The dual graph construction is linear time on the size of the $2-$ complex.


\end{abstract}


\end{frontmatter}


\section{Introduction}
Kirchhoff's laws for electrical networks state that the net current leaving a node is zero
(KCL) and the algebraic sum of the voltages around a loop is zero (KVL).
This topological model for electrical networks
has proved enormously useful both for theoretical studies and for practical computations.
Interest in these ideas in other areas of research is growing (see, for instance,
\cite{christiano}, \cite{madry}, 
\cite{ST}, 
\cite{spielman1}, \cite{spielman2}). 
It is therefore pertinent to explore whether there exist variations of this model,
which share essential characteristics with it.
In our opinion these essential characteristics are 
\begin{itemize}
\item that the spaces of vectors which satisfy Kirchhoff's voltage law and
Kirchhoff's current law are complementary orthogonal (Tellegen's Theorem \cite{belevitch}, \cite{book}, \cite{tellegen}, \cite{seshu});
\item that the preprocessing, needed for ease of solution, of constraints arising from these laws and from the device characteristic (eg. Ohm's law) can be done far more efficiently than if they
are treated as merely linear algebraic constraints.
\end{itemize}

Early work in the spirit of this paper, but which leads essentially to a graph based model,
is available in \cite{kron}. The graph based model was studied rigorously
for the first time in \cite{seshu}.
Both the general case where the network is regarded as a pair of complementary orthogonal
spaces with no connection to complexes and the case where there is an underlying graph have been treated
in \cite{book}.
This paper is about other situations where both are satisfied but which do not appear to have been considered
in the literature with the point of view of efficiency of equation formulation
as well as of preprocessing for ease of solution.

Three dimensional versions of Kirchhoff's laws have already been studied, for instance,
 in the case of magnetic circuits (for basic ideas see \cite{deltoro}) and
in the case of electromagnetic fields (for a  general and comprehensive description see \cite{branin1}, \cite{branin2}). However these problems
can be reduced to the case of graphs by a process of `cell duality'. Suppose we decompose
a large tetrahedron which encloses the region of interest into smaller tetrahedra
which intersect each other only in mutual faces.
 The relationship between these smaller tetrahedra can be captured by replacing each of these latter
by a node and joining tetrahedra which share a triangle by an edge.
This resulting graph could be called the cell dual of the original three dimensional complex.

Two dimensional versions of Kirchhoff's laws, which is the subject of this paper,  have features which appear essentially
different from the above. 
In order to describe the work in this paper and also to clearly differentiate it
from that available in the literature,
we begin by describing a possible generalization.
An $n-$complex (see Section \ref{prelim}) is made up of a series of $(j,j-1)$ skeletons, $0<j\leq n,$
each of which is made up of oriented $j-$cells and $(j-1)-$cells. The incidence relationship of this skeleton is captured by a 
$j-$coboundary matrix
with columns and rows corresponding to $j-$cells and  $(j-1)-$cells respectively. (A $(1,0)$ skeleton, for instance,
is a graph and its $1-$coboundary matrix is the incidence matrix of the graph.) 
With each $(j,j-1)$ skeleton one can associate a pair of complementary
orthogonal spaces. These are simply the row space of the $j-$coboundary matrix, i.e., the $j-$coboundaries,  and the space of
vectors orthogonal to the rows of the matrix, i.e., the $j-$cycles. If $A$ is the coboundary matrix, these are 
the space of vectors $y^T = \lambda ^TA $ and the space $x^TA^T=0.$ 
In \cite{vanderschaft1}, \cite{vanderschaft2}, higher dimensional electrical networks are defined
based on such skeletons, in the process unifying physical situations involving heat as well as electromagnetic fields.
Here, the first characteristic mentioned
above, namely complementary orthogonality, is obviously satisfied, but, in general, not the second characteristic of ease of writing and of preprocessing equations. 

In this paper, we consider a new situation where both characteristics are satisfied.
This involves $2-$complexes embedded in $\mathbb{R}^3,$ and, more generally, $(n-1)-$complexes embedded in $\mathbb{R}^n.$
The generalizations
of Kirchhoff's laws pertaining to these complexes that we use, 
are in the form of solution spaces of homogeneous `physical' equations.   
The complementary orthogonality is not obvious  but has to be derived.
Our approach reveals the connection of these ideas to contractibility of the underlying
space, complete unimodularity of the relevant vector spaces etc.
In addition, one of the main results of this 
paper is that we can build, in linear time on the size of the problem, 
an electrical (matroid) dual of the  relevant $2-$ complex that turns out to be  graph based.
The solution of this graph based electrical network yields the solution to the 
original $2-$complex based electrical network.

To motivate our study, we consider the problem of computing the magnetic intensity $H$ and the magnetic flux density $B$
in a three dimensional region when known current source loops exist in specified physical
locations. [Equivalent problems arise in situations involving  
electric field intensity $E$ and current density $J,$ and static problems involving temperature and heat.]

The relevant Maxwell's equations in the integral form are:
$$ \oint_C H.dl = \int_SJ.ds;\ \ \ \ \oint B.ds =0.$$
Here the surface $S$ is bounded by the contour $C.$ A clockwise traversal 
around it in, say, the  plane of the paper, would mean that the $ds$ vector 
on the right side of the first equation is directed into the plane.
The path integral, $ \int_P H.dl$  along a directed path $P,$ is called the magnetomotive force (mmf)
across the path.
The first equation states that the net mmf  around the 
contour is equal to the net current  passing through any surface bounded by the
contour.
The second equation states that the net flux leaving any closed surface is zero.
The material property is captured by a relation 
$$ B=\mu H.$$
This relates the two vectors through the `permeability' $\mu$ which could vary from point to point.
Suppose we have a cylinder with a cross sectional area $A,$ length $l$
and uniform permeability $\mu$ then the net flux $\phi$ passing through the cylinder 
in the direction of the axis is related to the mmf $m'$ along the axis by the equation
$$ m' = [\frac{l}{A\mu}]\phi;$$
The quantity $\frac{l}{A\mu}$ is called the reluctance of the cylinder in the direction of the axis.

In our problem, the medium is assumed to be composed of high permeability material embedded with 
thin layers of varying low permeability. Such a  situation might arise, 
for instance, when magnetic material develops cracks due to degradation.
This problem can be solved by solving partial differential equations over three dimensional
regions but that method provides poor insight while being computationally intensive.
We however, choose to model this essentially as a two dimensional complex embedded in $3-$space.

We  describe here the general discretized version of this  problem and later 
in Section \ref{sec:simple}, discuss an elementary instance.
In the general discretized version, we may imagine a bounded tetrahedral region in $\mathbb{R}^3,$
 being decomposed into smaller tetrahedra. Except for some previously specified facets (triangles) of these tetrahedra, the permeability
everywhere may be taken to be infinite. 
As an idealization, interiors of the tetrahedra
have zero reluctance and zero conductance and carry neither mmf
nor electric current,
the triangular facets may be taken to have 
zero thickness but can have a positive mmf across (normal) to them and therefore a positive reluctance
along the normal, 
the interior of the triangles may be taken to have no currents while
some of the boundaries of triangles may carry loops of currents,
and except for the triangles specified, the reluctance everywhere may be taken to be zero.

The constraints are 
\begin{itemize}
\item the net flux leaving a closed surface  (i.e., the boundary of a bounded region) is zero;
\item the net mmf,  through the set of triangles incident at an edge, taken say clockwise around the edge
with respect to its direction, is equal to the current through the edge in that direction (see Figure 
\ref{fig:trianglesedge1}); this constraint arises from the Maxwell equation 
$ \oint_C H.dl = \int_SJ.ds,$ where there is nonzero contribution to the integral in the left hand  side
only when the contour crosses a triangle.

\begin{figure}
 \label{fig:trianglesedge1}
\centering
 \includegraphics[width=2in]{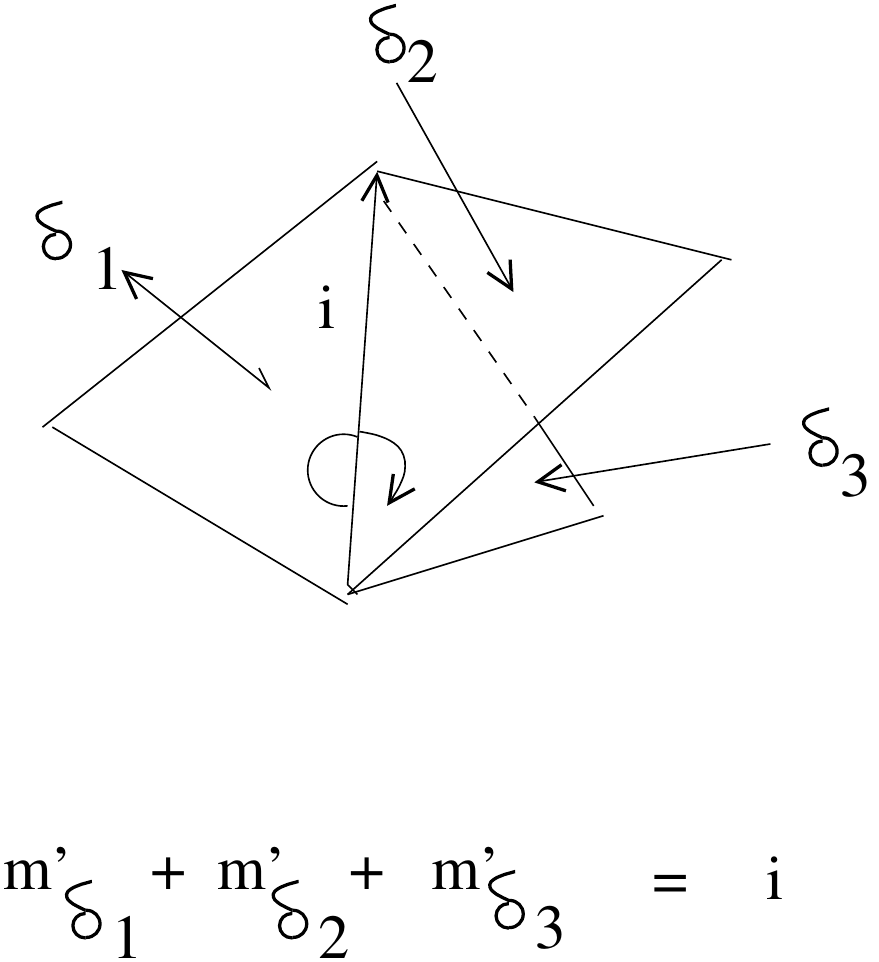}
 \caption{Triangles around an edge}
\end{figure}
\item for each triangle, the mmf $m'$ through triangle = flux $\phi $ through triangle $\times $ reluctance of the triangle.
\end{itemize}

The first two constraints are topological, analogous to Kirchhoff's laws.
We show in this paper, that  in the case of a $2-$ complex in the interior of a bounded tetrahedral region,
that the two sets of constraints, when the right side is zero, yield a pair of complementary orthogonal spaces.
This corresponds, in the case of $1-$dimensional electrical circuits on graphs, 
to voltage and current spaces being complementary orthogonal.

The outline of the paper follows.

Section \ref{sec:simple}
illustrates the physical problem through a simple example in which the triangles of interest lie
on the surface of a cube with one of them carrying a current in its boundary.

Section \ref{prelim} is on preliminary definitions on simplicial complexes, chains, cycles 
and coboundaries, boundaries of chains, coboundary
matrices  associated with the complexes etc.
It is shown that the fact that boundary of a boundary is a zero chain is equivalent
to product of $(j-1)-$coboundary matrix and $j-$coboundary matrix being the zero matrix.
It is shown that  $j-$coboundaries and $j-$cycles of a complex
  form complementary orthogonal spaces. 

Section \ref{sec:2complex}
deals with the special case of a $2-$complex embedded in $\mathbb{R}^3.$
Here, it is shown that the $2-$coboundary and $2-$cycle spaces are 
regular (completely unimodular) and,
by the use of a basic contractible space theorem, that the row space of 
the $2-$coboundary matrix is complementary orthogonal to the column space of the 
$3-$coboundary matrix of the complex. 

Section \ref{sec:electrical} is on electrical $2-$networks defined on $2-$complexes.
We prove a generalization of the celebrated Tellegen's theorem of electrical networks 
for electrical $2-$networks.
We also explicitly state the equations for such networks and show how to solve them
in the case of linear networks. In particular this shows how the physical problem
stated in the introduction can be solved.

Section \ref{sec:triangle} is on the notion of a triangle adjacency graph $tag(\C)$ for the $2-$complex $\C.$
This graph has two vertices $v_+,v_-$ for each triangle of $\C.$ If the triangle were on the $x-y$
plane in $\mathbb{R}^3,$ one of these would be above and the other below at a distance $\epsilon .$
Two vertices are joined by an edge in $tag(\C),$ if the corresponding triangles of $\C$ have a common edge
and one can slide from one to another along triangles. For instance, if triangles $\delta_1, \delta_2, \delta_3$ 
of $\C$ were as in Figure \ref{fig:trianglesedge1}, then in $tag(\C),$ 
$v_+(\delta_1) $ would be connected by an edge to $v_-(\delta_2).$
The graph  $tag(\C)$  can be constructed in time linear in size
of $\C.$
It is shown that two vertices of $tag(\C)$ can be connected by a path iff
physically, the two vertices are in the same connected region of $\mathbb{R}^3\setminus\C.$
The graph $\G_{comptag(\C)}$ is built on connected components of vertices of $tag(\C).$
There is an edge  corresponding to each triangle of $\C$ directed from the component
of $tag(\C)$ containing $v_+$ to the one containing $v_-$ of the triangle.
Finally, it is shown that the rows of the incidence matrix of $\G_{comptag(\C)}$
are $2-$cycles of $\C.$

Section \ref{sec:celldual} is on the cell dual of a $3-$complex $\C_{\T}$ which is the decomposition
of a tetrahedron $\T$ into smaller tetrahedra $\tau_i$ whose interiors do not intersect.
The triangles which are the faces of these tetrahedra are either in the boundary of
$\T,$ in which case they belong to only one of the  $\tau_i$ 
or are common to exactly two of the $\tau_i.$ It is assumed that the given $2-$complex $\C$ lies in the interior
 of $\T$ and has its triangles as a subset of the triangles of $\C_{\T}.$ The cell 
dual is the graph $\G_{\T}$ which 
is obtained by replacing each 
$\tau_i$ by a node and joining two nodes if they represent tetrahedra which share a common triangle.
A region graph $\G_{region(\C)}$ corresponding to connected regions of $\T\setminus\C$ 
is shown to be obtained from $\G_{\T}$ by contracting edges which correspond to
triangles which do not belong to $\C.$
The row space of the incidence matrix of this region graph is shown to be the $2-$cycle space of $\C.$

Section \ref{sec:region=tag} is on the proof of the fact that when $\C$ is connected
$\G_{region(\C)}$ is identical to $\G_{comptag(\C)}.$ Since the latter can be constructed in time 
linear on the size of $\C,$ from the results of the previous section we have a convenient 
and easily constructed representation for the $2-$cycle space of $\C.$
This is sufficient for writing a linearly independent set of equations 
for the linear electrical $2-$network.

Section \ref{sec:matroid} discusses the matroid duality of the complex $\C$ and the graph
$\G_{region(\C)}$ and also sketches, given a non planar graph, how to enlarge it so that 
it has a matroid dual which is a $2-$complex.

Section \ref{sec:dualnetwork} describes how to build the dual network to a given $2-$network so as
to infer the solution of the $2-$network from that of the dual, and also 
indicates how to translate results from graph based networks to $2-$networks by considering the 
case of the Kirchhoff's tree formula.

Section \ref{sec:gen} sketches how to generalize the ideas of the paper from $2-$complexes
embedded in $\mathbb{R}^3$ to $(n-1)-$complexes embedded in $\mathbb{R}^n.$

Section \ref{sec:conclusion}
is on conclusions.
\section{A simple instance of the problem}
\label{sec:simple}
We now illustrate the general ideas through a simple example.

In Figure \ref{fig:cube}, except for the thin layers 
of the faces, both the inside and outside of the cube are of high permeability ($\infty ,$
for simplicity).  
\begin{figure}
\centering
 \includegraphics[width=1.5in]{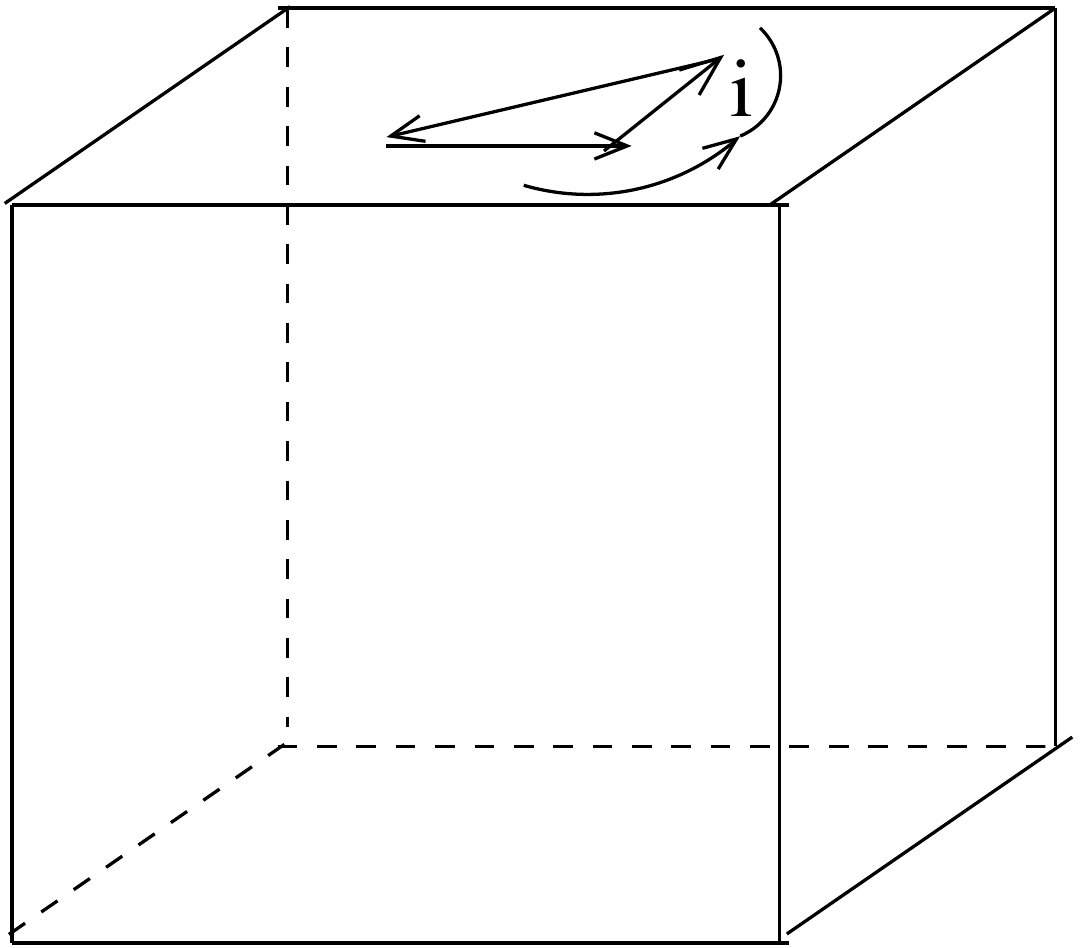}
 \caption{Cube with faces of low permeability}
 \label{fig:cube}
\end{figure}
We  divide the faces into triangles. 
One of the triangles would have a loop carrying current $i,$
in the direction of the orientation of the triangle,
on the upper face of the cube. 
The triangles should be small enough that one can take
the thickness to be constant  and the permeability to be  uniform in its interior without causing unacceptable error.
After computing the reluctance of the triangles, we will idealize them to have zero thickness.

The solution procedure is as follows.
\begin{enumerate}
\item Orient all the triangles consistent with  the outward normal of the cube (an orientation of a triangle
is an ordering of its vertices - see the beginning of Section \ref{prelim} and Figure \ref{fig:triangle_orientation}).  
\item For the $k^{th}$ triangle, compute the reluctance $r_k = \frac{d_k}{\mu _kA_k},$
where $\mu _k$ is the permeability in the thin triangular slab, $A_k$ is the area
and $d_k,$ the thickness. Let $H_k, B_k$ be the average magnetic field intensity and the 
average flux density normal to the surface of the triangle.
We write
$$H _k d_k= B_kA_k\times \frac{d_k}{\mu_k A_k} \ \ \mbox{i.e.,} $$ 
\begin{align}
\label{reluctance}
m'_k = \phi _k \times r_k \ \ \mbox{(mmf= flux times reluctance)}. 
\end{align}
\item With the orientation specified we have constraints for the fluxes and mmfs associated with triangles
($k^{th}-$ flux, $k^{th}-$ mmf with $k^{th}-$ triangle) as follows.
If $\phi $ denotes flux, we must have net flux leaving every closed surface equal to zero.
In the present case there is only one closed surface.
\begin{align}
\label{flux}
\sum_k \phi _k & =  0\ \ \mbox{i.e.,}\\
\bbmatrix{1 &\cdots & 1}\bbmatrix{\phi _1\\\vdots\\\phi_n} &= 0.
\end{align}
The constraint on the mmf vector is that the net mmf around an edge (i.e., the sum of the mmfs associated with the triangles
incident at an edge) is the current through the edge in the direction consistent 
with the direction of rotation.
This constraint can be expressed in terms of the rows  of the $2-$coboundary matrix
which has columns corresponding to triangles and rows corresponding to edges.
In the row for edge $e,$ the entry for triangle $\delta $
is $0$ if it is not incident on the edge,
is $+1$ if it is incident and agreeing with
 the orientation of the edge and 
is $-1$ if it is incident but oppositely oriented
to the orientation of the edge.

In the present case, every edge is incident on exactly two triangles and if we were to orient the triangles according to
the outward normal of the cube, in one it would agree with the orientation and,  in the other, it would oppose it.
Let  $A^{(2)}$ denote the $2-$ coboundary matrix and $m_k,$ the $k^{th}$ mmf. Further, let the first three rows
$A^{(2)}_{1,.},A^{(2)}_{2,.},A^{(2)}_{3,.}$ correspond to edges of the triangle carrying the current $i,$
edges oriented in the direction of the current
and let the remaining rows be denoted $A^{(2)}_{rem,.}.$
\begin{align}
\label{mmf}
\bbmatrix{A^{(2)}_{1,.}\\A^{(2)}_{2,.}\\A^{(2)}_{3,.}\\---------\\A^{(2)}_{rem,.}}\bbmatrix{m_1'\\\vdots\\m_n'}= \bbmatrix{i\\i\\i\\----\\{\bf 0}}.
\end{align}
\item Solve simultaneously Equations \ref{reluctance}, \ref{flux} and \ref{mmf}.
\end{enumerate}

The above appears as a straight forward linear algebraic problem but with singular equations.
The first step is to build an equivalent nonsingular set of equations. 
If, however, we use linear algebraic methods at this stage, it would be (relatively speaking) computationally expensive.
We will show how to do this 
combinatorially in linear time and also show that, after this, the problem reduces to solving a conventional 
electrical circuit based on graphs. This is one of the contributions of this paper. 
After that we would be left with a set of sparse linear equations
for which very efficient practical procedures exist in the literature.
\section{Preliminaries}
\label{prelim}
\begin{figure}
\centering
 \includegraphics[width=1.5in]{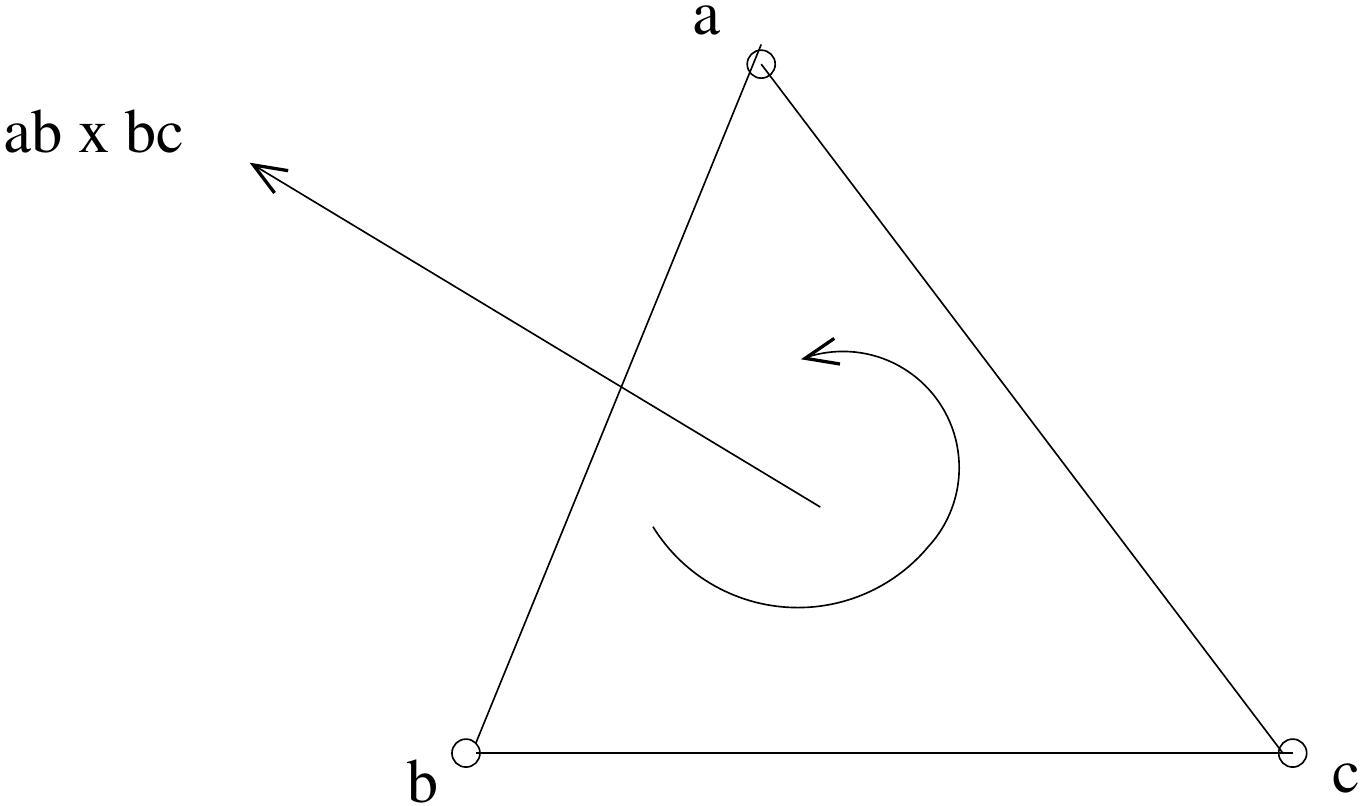}
 \caption{Triangle $abc$ oriented consistent with the vector product $ab\times bc$ }
 \label{fig:triangle_orientation}
\end{figure}
This section follows \cite{whitney} (more recent references are \cite{hatcher}, \cite{munkres}).

Let $a_0, \cdots ,a_n$ be vectors in $\mathbb{R} ^n.$ 
An $n-$simplex $a_0 \cdots a_n$
is the convex hull of the vectors $a_0, \cdots ,a_n,$ i.e., the set of all points 
$\sum_0^n\lambda _ia_i ,$ where $\lambda_i \geq 0, \sum_0^n\lambda _i=1.$
A face of the $n-$simplex $a_0 \cdots a_n$
is the convex hull of any subset of the vectors $a_0, \cdots ,a_n.$
An oriented $n-$cell 
$c$ is a simplex $a_0 \cdots a_n$
with vertices $a_0, \cdots ,a_n$ and with a specifed ordering of vertices $[a_0, \cdots ,a_n].$
Two orderings $[a_0, \cdots ,a_n], [b_0, \cdots ,b_n] $ of the vertices are treated as equivalent 
if $[b_0, \cdots ,b_n] $ is an even permutation of $[a_0, \cdots ,a_n]$ and opposite if an odd permutation.
We will usually call a $2-$cell, a triangle.
If $abc$ is  an oriented triangle embedded in $\mathbb{R}^3,$ we say the orientation $[abc]$ is consistent with the direction of the
vector cross product $ab \times bc$ (see Figure \ref{fig:triangle_orientation}).

Suppose $S:= \{c_1, \cdots , c_m\}$ is a set of oriented $n-$cells.
Let $c_j, c_j', j=1, \cdots ,m$ be oppositely oriented. Let $S':= \{c_1', \cdots , c_m'\}.$ We can think of a new orientation 
for the cells to be a one to one map $\sigma:S\rightarrow S\cup S'$ 
with $\sigma(c_j)= c_j \ \mbox{or} \ c_j'.$
\begin{example}
An oriented $0-$cell is a point. An oriented $1-$cell is a directed edge. An oriented $2-$cell  is an oriented
triangle. An oriented $3-$cell is an oriented tetrahedron (see Figure \ref{fig:regular}).
The $1-$cells $ab,ba$ are oppositely oriented. The $2-$cells $abc,cba$ are oppositely oriented and similarly the 
$3-$cells $abcd, dcba.$ The $3-$cells $abcd,badc$ denote the same oriented cell since
$\ppmatrix{a&b&c&d\\b&a&d&c}$ is an even permutation.
\end{example}

Let $c_1, \cdots , c_m$ be oriented $n-$cells. An $m-$chain is a formal sum $\sum_1^m \alpha_ic_i, \alpha_i \in \mathbb{R}.$ When $c_i$ is an $n-$cell with $n\geq 1,$ We take $-c_i$ to be 
the $n-$cell with the same vertices but with orientation opposite to that of $c_i.$
The formal sum is the `zero chain' iff each of the $\alpha_i$ is zero.
If $\sum_1^n \alpha_ic_i,\sum_1^n \beta_ic_i,$ are two $m-$chains, we define their sum to be 
$\sum_1^n (\alpha_i+\beta_i)c_i.$
We define $\lambda(\sum_1^n \alpha_ic_i)$ to be $\sum_1^n \lambda \alpha_ic_i.$
 
Given an oriented $n-$cell $a_0 \cdots a_n,$ the boundary $\partial(a_0 \cdots a_n)$
is defined to be  the $(n-1)-$chain 
$$\underline{a}_0a_1\cdots a_n + (-1)^1a_0 \underline{a}_1\cdots a_n+ \cdots + (-1)^ja_0 \cdots \underline{a}_j\cdots a_n + \cdots + (-1)^na_0 \cdots \underline{a}_n,$$
where $a_0 \cdots \underline{a}_j\cdots a_n$ denotes $a_0 \cdots a_{j-1} a_{j+1}\cdots a_n.$
When $a_0$ is a $0-$cell, we take $\partial(a_0):= 0.$ 

For any $m-$chain $\sum_1^m \alpha_ic_i,$ we take $\partial (\sum_1^m \alpha_ic_i):= \sum_1^m \alpha_i\partial (c_i).$
It follows from the definition that $\partial\partial(a_0 \cdots a_m)$
is the zero $(m-2)-$chain, for $m\geq 2.$ For $m<2,$ we take it to be trivially zero.
For any $m-$chain,  $\sum_1^m \alpha_ic_i,m\geq 2,$ we therefore have $\partial\partial (\sum_1^m \alpha_ic_i)$
as the zero $(m-2)-$chain and for $m<2,$ we take it to be trivially zero. 
\begin{example}
$\partial\partial(a_0 a_1 a_2)= \partial(a_1a_2-a_0a_2+a_0a_1)=(a_1-a_2)-(a_0-a_2)+(a_0-a_1)=0.$
\end{example}
Cochains are defined to be linear functionals acting on chains.
We will deal with chains and cochains uniformly using the notion of a vector.
\begin{definition}
\label{def:vector}
A vector is a function $h:S\rightarrow \F,$ from a finite set $S$ to a field $\F.$ We say 
that the vector is on $S$ over $\F.$ In the present paper 
the field  would be  $\mathbb{R}.$ The support of $h,$ denoted $supp(h),$ is the 
subset of $S,$ where $h$ takes nonzero values. 
Addition $h_1+h_2$ of vectors $h_1,h_2$ on $S$ over $\F$ is defined by $(h_1+h_2)(e):=h_1(e)+h_2(e), \ e\in S;$
scalar multiplication $\lambda h_1,\ \lambda \in \F,$ of $h_1$  is defined by $(\lambda h_1)(e):=\lambda (h_1(e)), \ e\in S.$
A vector space on $S$ over $\F$ is a collection of vectors $S$ over $\F$ closed under addition and scalar multiplication.
 
\end{definition}
In the case of chains we can identify $\sum_1^m \alpha_ic_i$ with the vector $h$ such that 
$h(c_i)= \alpha_i.$ If $c_i'$ is oppositely oriented to $c_i$  we take $h(c_i')= -\alpha_i.$
Thus if $S:=\{c_1, \cdots ,c_m\},$ $\alpha c_1+\beta c_2 + \gamma c_3$ can be identified with
$\begin{matrix}
\  &c_1&c_2&c_3&c_4&\cdots &c_m&\  
\\(&\alpha &\beta & \gamma&0&\cdots&0&)
\end{matrix}
$ .

\begin{definition}
An $n-$complex $\C,$ is a collection of $n-$cells, $(n-1)-$cells, $\cdots $, $0-$cells,
such that every face of a cell in $\C$ is also in $\C$ and the intersection of any two cells 
of $\C$ is a face of each of them.
Thus, the cells in the boundary of each $j-$cell, $0<j\leq n,$ are also present in the complex.
Further, a pair of distinct $k-$ cells, $0<k\leq n$ either has no intersection
or intersects in a $j-$cell, $j\leq k-1,$ present in $\C .$
We will denote the collection of oriented $k-$cells of $\C$ by $S^{(k)}(\C)$ or, when clear
from the context, by $S^{(k)}.$
\end{definition}
We only deal with oriented complexes and will omit the prefix `oriented' while referring to them.

To define higher dimensional electrical circuits, we need the definition of a $(k,k-1)$ skeleton of 
a complex.
\begin{definition}
The $(k,k-1)$ skeleton, $1\leq k\leq n,$ of an $n-$complex $\C$
is the pair $(S^{(k)}(\C),S^{(k-1)}(\C))$
with the boundary relationship between the $k-$cells  in $S^{(k)}$ and the $(k-1)-$cells
in $S^{(k-1)}$ agreeing with that of the complex $\C.$
\end{definition}
The most natural way of defining new complexes from old is through the notion of a subcomplex.
\begin{definition}
Let $\C$ be an $n-$complex and let $k\leq n.$ Let $T\subseteq S^{(k)}(\C).$ The $k-$subcomplex 
$\C_k$ of $\C$ on $T$ has $S^{(k)}(\C_k)= T$ and for $j\leq k$ has 
$S^{(j-1)}(\C)\supseteq S^{(j-1)}(\C_k) $ 
with the
 boundaries of $j-$cells of $\C_k$ 
agreeing with the boundaries of these cells in $\C.$  
\end{definition}
The `connectedness' of complexes, defined below, has a role in algorithms developed in the present paper.
\begin{definition}
\label{def:connect}
An $n-$complex $\C$ is said to be $n-$connected (or connected in brief), if given any pair of $(n-1)-$cells in the complex, $c_1, c_{end},$
there exists a sequence $c_1,\cdots , c_k, c_{end}$ of $(n-1)-$cells such that  $c_i,c_{i+1}, 1\leq i< k,$
are faces of a  common $n-$cell  and so are $c_k, c_{end}.$
A graph is a $1-$complex.
A $1-$subcomplex of a graph is called a subgraph.
The connected components of a graph are its maximally connected subgraphs.
\end{definition}
\begin{definition}
A (path) connected region  in $\mathbb{R}^n$ is a set of points in $\mathbb{R}^n$ with the
property that given any two points in the set, it contains a continuous path between them.
\end{definition}

A convenient way of representing an $n-$complex $\C,$ is through a sequence of coboundary matrices
$A^{(1)}(\C), \cdots , A^{(n)}(\C).$

\begin{definition}
\label{def:coboundary}
The matrix $A^{(k)}(\C), 1\leq k\leq n,$ has its columns indexed by the $k-$cells of $\C,$ say
$c^{(k)}_1, \cdots , c^{(k)}_r,$ and rows indexed by the $(k-1)-$cells, say 
$c^{(k-1)}_1, \cdots , c^{(k-1)}_m.$ 
The $(i,j)^{th}$ entry of this matrix would be zero, if $c^{(k-1)}_i$ does not lie in the 
When clear from the context we will write $A^{(k)}$ in place of $A^{(k)}(\C).$ 
boundary of $c^{(k)}_j,$  $+1, $ if it occurs with a positive sign in the boundary and 
$-1,$ if it occurs with a negative sign in the boundary.
$k-$cells can be degenerate in the  sense that their boundary is zero or made up
of  $p-$chains where $p\leq k-2.$
In such a case the corresponding column of $A^{(k)}(\C)$ would be the zero column.
\end{definition}
When clear from the context we will write $A^{(k)}$ in place of $A^{(k)}(\C).$ 

It is evident that the column $c^{(k)}_j$  represents the boundary vector 
$\partial(c^{(k)}_j):= \sum \beta_ic^{(k-1)}_i$ 
with the $(i,j)^{th}$ entry being $\beta _i.$
\begin{example}
Let $\C$ be a $3-$complex and let $abc$ be an oriented $2-$cell and let $ab,cb,ac$
be oriented $1-$cells in $\C.$ Then, in $A^{(2)}(\C),$
the column $abc$ will have $+1$ in the row $ab,$
 $-1$ in the row $cb,$
 $-1$ in the row $ac$
and in all other rows will have $0.$
\end{example}
Consider the $k-$chain 
$\sum_1^n x_jc_j^{(k)}.$ The boundary  $\partial (\sum_1^n x_jc_j^{(k)})= \sum_1^n x_j\partial (c_j^{(k)}).$
Therefore this vector is represented in terms of $(k-1)-$cells of the complex by
$\bbmatrix{A^{(k)}(\C)}\bbmatrix{\bf{x}}.$
Since the boundary of any boundary vector is the zero chain, we must have
$\bbmatrix{A^{(k-1)}(\C)}\bbmatrix{A^{(k)}(\C)}\bbmatrix{\bf{x}}=\bbmatrix{\bf{0}},$
for any vector $\bf{x}.$
It follows that $\bbmatrix{A^{(k-1)}(\C)}\bbmatrix{A^{(k)}(\C)}=\bbmatrix{\bf{0}},k\geq 2,$
i.e., every row vector of $\bbmatrix{A^{(k-1)}(\C)}$ is orthogonal to every column vector
of $\bbmatrix{A^{(k)}(\C)}.$ On the other hand, if $\bbmatrix{A^{(k-1)}(\C)}\bbmatrix{A^{(k)}(\C)}=\bbmatrix{\bf{0}},$ it is clear that $\bbmatrix{A^{(k-1)}(\C)}\bbmatrix{A^{(k)}(\C)}\bbmatrix{\bf{x}}=\bbmatrix{\bf{0}},$
which implies that  boundary of any boundary vector is zero. Thus the statement
that the boundary of the boundary of a $k-$chain is the zero $(k-2)-$vector is
equivalent to $\bbmatrix{A^{(k-1)}(\C)}\bbmatrix{A^{(k)}(\C)}=\bbmatrix{\bf{0}}.$
Note that this only implies that the row space of 
$\bbmatrix{A^{(k-1)}(\C)}$ and the column space of $\bbmatrix{A^{(k)}(\C)}$
are orthogonal and not that they are complementary orthogonal.
The situation relevant to this paper is where they are actually complementary orthogonal
 (i.e., where the `$(k-1)^{th}$ homology group' is zero).

Let $S^{(k)}, 0\leq k\leq n$ denote the set  of $k-$cells of the $n-$complex
$\C.$ Given a vector $y:S^{(k)}\rightarrow \mathbb{R},$ we define the coboundary $z$
of $y$ by 
$\bf{z}^T:= \bf{y}^T\bbmatrix{A^{(k)}}$
 (treating $z,y$ as  row vectors $\bf{z}^T,\bf{y}^T$).

Since $\bf{y}^T(\bbmatrix{A^{(k)}}\bbmatrix{\bf{x}})=(\bf{y}^T\bbmatrix{A^{(k)}})\bbmatrix{\bf{x}},$
it is clear that the `action' of  $y$ on the $(k-1)-$chain
represented by the boundary vector $\bbmatrix{A^{(k)}}\bbmatrix{\bf{x}}$
is the same as the `action' of the  coboundary of $y$  on the $k-$chain 
$\sum_1^n x_jc_j^{(k)}$ represented by the vector ${\bf{x}}.$
We call a  vector belonging to the row space of $\bbmatrix{A^{(k)}(\C)}$ 
a $k-$coboundary of $\C$ 
and a $k-$chain $\sum_1^n x_jc_j^{(k)},$ a $k-$cycle of $\C,$
if its boundary is zero, i.e., satisfies
$\bbmatrix{A^{(k)}(\C)}\bbmatrix{\bf{x}}=\bbmatrix{\bf{0}}.$

We summarize the above discussion in the following theorem.
\begin{theorem}
\label{cycle_cob}
For any $n-$complex $\C,0<k\leq n,$ 
\begin{enumerate}
\item the boundary of a $k-$chain $x$ is $\bbmatrix{A^{(k)}}\bbmatrix{\bf{x}};$
\item every boundary of a  $k-$chain
is a $(k-1)-$cycle;  
\item $\bbmatrix{A^{(k-1)}(\C)}\bbmatrix{A^{(k)}(\C)}=\bbmatrix{\bf{0}};$
\item the space of $k-$coboundaries and the space of
$k-$cycles are complementary orthogonal.
\end{enumerate}
\end{theorem}

\section{$2-$complexes in $\mathbb{R}^3$} 
\label{sec:2complex}
\begin{figure}
\centering
 \includegraphics[width=1.5in]{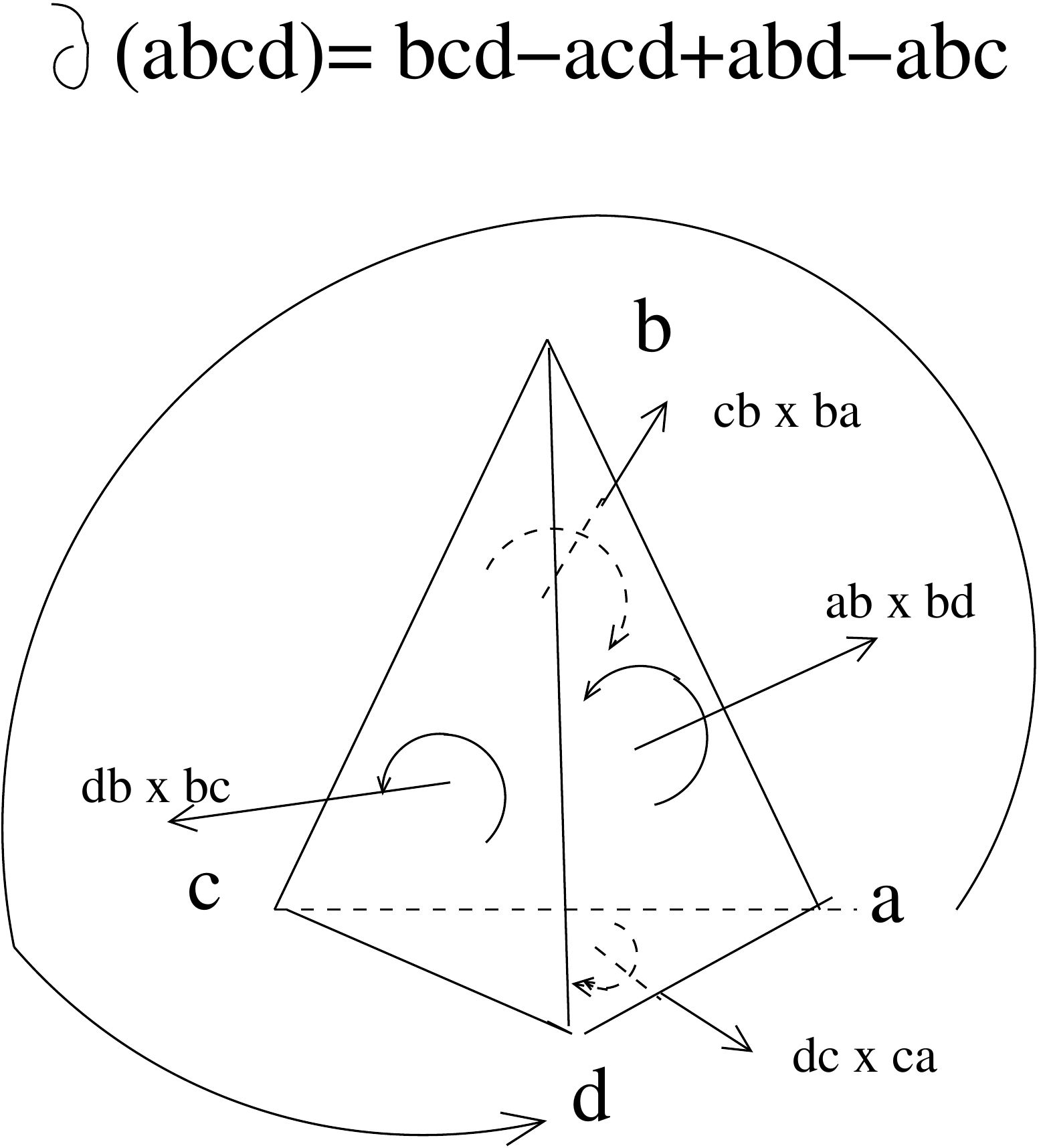}
 \caption{Regular orientation of a tetrahedron}
 \label{fig:regular}
\end{figure}

In this paper our interest is in $(2,1)$ skeletons of $3-$complexes which are embedded 
in $\mathbb{R}^3$ since we build our higher dimensional electrical networks on such 
skeletons. We show in this section that the $2-$coboundaries and $2-$cycles associated with these skeletons have
certain properties, for instance complete unimodularity,  which are useful for our purpose. \\

We are given a $2-$complex $\C$ embedded in $\mathbb{R}^3,$ i.e., the vertices of triangles
($2-$cells) of $\C$ are vectors in $\mathbb{R}^3.$
We consider the region of interest to be a tetrahedron $\T$
which contains $\C$ in its interior.
The tetrahedron $\T$ is decomposed into a 
set of tetrahedra whose interiors do not intersect 
and which together with their faces constitute the $3-$complex $\C_{\T}.$
Additionally, the decomposition is such that
the set of triangles $S^{(2)}(\C)$ of $\C$ is contained in the set of triangles
$S^{(2)}(\C_{\T})$ of $\C_{\T}.$
We begin by assigning an orientation to $\T.$
Using the orientation of $\T,$ as the reference, we can orient all the tetrahedra 
of $\C_{\T}$
consistently so that if two of them intersect in a triangle, the boundaries of the 
two tetrahedra assign opposite orientations to the triangle.
We could, without loss of generality, assume this orientation of the triangle
to be consistent with the outward 
normal with respect to the tetrahedron.
Let us call this orientation, a regular orientation for the tetrahedra
of $\C_{\T}$ (see Figure \ref{fig:regular}).\\

We will work with regions which are unions of tetrahedra
of $\C_{\T},$
which we will call regions of $\C_{\T}.$
Let $P$ be a subset of regularly oriented tetrahedra of  $\C_{\T}.$
The chain $\sum_{\tau_i\in P}\tau_i$ can be identified with the 
region $\Reg_P:= \bigcup_{\tau_i\in P}\tau_i.$
We then have 
\begin{theorem}
\label{boundary_of_region}
The boundary of $\Reg_P,$ i.e., the formal sum of the oriented boundary triangles of $\Reg_P,$ is equal to the $2-$cycle $\sum_{\tau_i\in P}\partial(\tau_i).$
\end{theorem}
\begin{proof}
When 
tetrahedra $\tau_i,\tau_j, i\ne j$ intersect in a triangle, because of the regular
orientation of tetrahedra, this common triangle
is oppositely oriented in $\tau_i,\tau_j,$ and so cancels out in $\sum_{\tau_i\in P}\partial(\tau_i).$
The only terms that remain are those triangles which occur only once as the boundary of a tetrahedron in $P.$ These are precisely the triangles at the boundary of $\Reg_P$
and their orientation would be consistent with the outward normal of the tetrahedron
in question and therefore also with the outward normal with respect to the region
$\Reg_P.$
Since we have 
$\partial(\sum_{\tau_i\in P}\partial(\tau_i))= \sum_{\tau_i\in P}\partial\partial(\tau_i)=0,$
it follows that $\sum_{\tau_i\in P}\partial(\tau_i)$ is a $2-$cycle
of $\C_{\T}.$
Thus the formal sum of the oriented boundary triangles of the region $\Reg_P$ is equal to the 
$2-$cycle $\sum_{\tau_i\in P}\partial(\tau_i).$ 
\end{proof}
We now state a basic result in algebraic topology (Theorem 19.5 of \cite{munkres})
after a preliminary definition.
\begin{definition}
Let $\U$ denote the closed interval $[0,1].$
Let ${\C}$ be an $n-$complex in $\mathbb{R}^m, m\geq n,$ and let $\hat{\C}$ be the union of all $n-$cells of ${\C}.$
We say $\hat{\C}$  is contractible if there exists a continuous 
map $h: \hat{\C}\times \U\rightarrow \hat{\C},$  such that \\
for all $y\in \hat{\C},$
$h(y,0)=y,$ 
\\
for all $y\in \hat{\C}$ and for some $x\in \hat{\C},$  $h(y,1)=x$  and, further,
\\
for this $x,$  $h(x,z)=x, $  for all $z\in I.$ 
\end{definition}
\begin{theorem}
\label{Hurewicz}
If the union of  $n-$cells of an $n-$complex $\C$
is contractible, then \\
for $0<j\leq n,$ every $(j-1)-$cycle  
of the complex is the boundary of a $j-$chain.
\end{theorem}
If an $n-$ complex has the property that for $0<j\leq n,$ every $(j-1)-$cycle
of the complex is the boundary of a $j-$chain, then it is said to be acyclic.

Since every convex set in $\mathbb{R}^n$ is contractible it follows that the 
tetrahedron $\T$ in $\mathbb{R}^3$ is contractible.
We thus have  the following result about $2-$cycles and boundaries of $3-$chains.
\begin{theorem}
\label{Hurewicz2}
A $2-$chain $x$ of $\C_{\T}$ is a 
$2-$cycle of $\C_{\T}$ iff it is the boundary of a $3-$chain of $\C_{\T}.$ 
Equivalently, the row space of the $2-$coboundary matrix $A^{(2)}(\C_{\T})$ and the column space of 
the $3-$coboundary matrix $A^{(3)}(\C_{\T})$ are complementary orthogonal.
Hence the column space of
the $3-$coboundary matrix $A^{(3)}(\C_{\T})$ is the $2-$cycle space of $\C_{\T}.$
\end{theorem}
\begin{remark}
A direct proof of Theorem \ref{Hurewicz2} through induction on the number of regions
into which the tetrahedron $\T$ is divided is routine. However the link to 
Theorem \ref{Hurewicz} appears insightful. See Section \ref{sec:gen}
for an enlargement of the ideas involved.
\end{remark}
\begin{definition}
Let $\V$ be a vector space on $S$ over $\mathbb{R}.$
We say that a vector $x\in \V$ has minimal support (see Definition \ref{def:vector})
iff the support is nonempty and no other $z\in \V$  has its support properly
contained in that of $x.$
\end{definition}

We now have the following useful result about $\C_{\T}.$
\begin{theorem}
\label{minimal_2cycle}
Let the tetrahedra of $\C_{\T}$ be regularly oriented.
Let $x$ be a $2-$cycle of $\C_{\T}$ with minimal support. Then,
\begin{enumerate}
\item $x$ is the boundary of a $3-$chain $\tilde{y} := \sum \alpha_i\tau_i,$ 
where all the nonzero $\alpha_i$ are equal to a constant;
 \item 
$x$ is the multiple of a $0, +1,-1$ vector and 
the support of $x$ is the set of boundary triangles of the union of tetrahedra in the 
support of  $\tilde{y}. $
\end{enumerate}
\end{theorem}
\begin{proof}
1. By Theorem \ref{Hurewicz2}, we know that $x=\partial y$ for some $3-$chain of $\C_{\T}.$
Let the support of $y$ be $P$
and, without loss of generality, let the maximum entry of $y$ be positive and equal to, say, $k.$
Consider the vector $\tilde{y}$ defined by
$\tilde{y}(\tau)=y(\tau),$ 
if $y(\tau)=k,$ and 
$\tilde{y}(\tau)=0,$ otherwise.

We claim that the
support of the $2-$cycle $\partial \tilde{y}$ is contained in the support of
 $\partial {y}.$
To see this, we first observe that the support $\tilde{P}$ of $\tilde{y}$ is a set of regularly oriented tetrahedra in $\C_{\T},$
 whose union is the region $\Reg_{\tilde{P}}.$ 
Suppose a triangle $\delta $ occurs in the support of
$\partial \tilde{y}.$ Then by Theorem \ref{boundary_of_region}, it is a boundary triangle 
of $\Reg_{\tilde{P}}.$
Therefore it occurs as a facet of only one tetrahedron, say $\tau_{in}$ in $\tilde{P}.$ 
It can meet atmost one other tetrahedron  
in $P\setminus\tilde{P}.$ 
If it meets none,  $\delta $ will continue to occur  in the support of $x=\partial y.$
Suppose it meets one other tetrahedron say $\tau_{out}$
in $P\setminus\tilde{P}.$ 
Now $k=y(\tau_{in})> |y(\tau_{out})|.$ So $\partial y(\delta) = k\pm y(\tau_{out})\ne 0.$
In other words, $\delta $ occurs in the support of $x=\partial y.$  
But we know  that $x$ is a $2-$cycle of $\C_{\T}$ with minimal support.
We conclude that $y=  \tilde{y}.$ Now $\tilde{y}=k\hat{y},$ where $\hat{y}$
is a $0,+1$ vector and $x= k\partial \hat{y}.$\\

2.
Since $\hat{y}$
is a $0,+1$ vector, in $\partial \hat{y},$ the boundary triangles of 
$\Reg_{\tilde{P}}$  will acquire value $0,\pm 1,$ and the other triangles
which are common to two tetrahedra in $\tilde{P}$ will occur with 
opposite orientations in those two 
and will cancel out. 
Thus $\partial \hat{y}$  is a $0,\pm 1$ vector which makes
$x$ a multiple of a $0,\pm 1$ vector and it is clear that the support of 
$x$ is the set of boundary triangles of $\Reg_{\tilde{P}}.$
\end{proof}
\begin{definition}
A vector space $\V$ on $S$ over $\mathbb{R}$  is said to be regular or completely 
unimodular iff every minimal support vector is a multiple of a 
$0,\pm 1$ vector.  
\end{definition}

In the following theorem, we present a few well known and useful facts about regular vector spaces.
But first we need a few definitions.
\begin{definition}
A matrix is said to be totally unimodular iff all its subdeterminants are 
$0,\pm 1.$
\end{definition}
\begin{definition}
A matrix whose rows form a basis for a vector space $\V$ is called 
a representative matrix for $\V.$
If, after column permutation, it can acquire the form 
$\bbmatrix{I&|&K},$ where $I$ denotes an  identity matrix of the appropriate order, we say it is a standard representative matrix for $\V.$
The standard representative matrix is said to be with respect to the column subset 
$B,$ if this latter is the set of  columns of an identity matrix.
\end{definition}
\begin{definition}
Let $S$ be a finite set and $\mathcal{I}$ be a family of subsets of  $S$
such that maximal members of $\mathcal{I}$ contained in any subset of $S$ have the same size.
Then $\M:= (S,\mathcal{I})$ is called a matroid on $S.$
Members of $\mathcal{I}$ are called independent sets of $\M.$
A maximal independent set of $\M$  contained in $S$ is called a base of $\M.$
Complements of bases of $\M$ are called cobases of $\M.$
\end{definition}
The most natural example of a matroid is the pair $(S,\mathcal{I}),$ where $S$ 
is the set of columns of a matrix $A,$ and $\mathcal{I}$ is the family of independent
subsets of columns of $A.$
We denote this matroid by $\M(A).$

It is clear that a set of columns of a representative matrix of $\V$ is independent
iff the corresponding set of columns is independent in any other
representative matrix of $\V.$

\begin{definition}
\label{def:vectormatroid}
Let $\V$ be a vector space on $S$ over $\mathbb{R}.$ Let columns of a representative matrix $A$
of $\V$ be identified with elements of $S.$ 
The matroid $\M(A)$ (which is independent of the representative matrix $A$ chosen for $\V$)
is called the matroid $\M(\V)$ associated with $\V.$


\end{definition}
%
%
The following elementary result can be proved using ideas from
matroid theory or by translating these ideas to vector spaces.
We omit the proof.
\begin{lemma}
\label{base_circuit}
If $x$ is a minimal support vector in $\V,$ then there exists a base
of $\M(\V),$ which intersects $supp(x)$ in a single element.
\end{lemma}
\begin{definition}
\label{def:minor}
Let $x,y$ be vectors on $S$ over a field $\F.$ Then the dot product 
$<x,y> := \sum_{e\in S}x(e)y(e).$ We say $x,y$ are orthogonal iff $<x,y> =0.$
Let $\K$ be a collection of vectors on $S.$ 
$\K^{\perp}:= \{y:<x,y>=0, x \in \K\}.$
Let $T\subseteq S.$ Then $\K \circ T:= \{x_T:x_T=y|_T, y \in \K\}$
is called the restriction of $\K$ to $T$
and\\
 $\K \times T:= \{x_T:x_T=y|_T, y \in \K, y(e)=0, e\notin T\}$
is called the contraction of $\K$ to $T.$
\end{definition}
\begin{theorem}
\label{minimal_support}
Let $\V$ be a vector space on $S$ over $\mathbb{R}.$
Let $Q_B$ be a standard representative matrix with respect to a base $B$
of $\M(\V).$ Let $e\in B$ and let $x_e$ be the row of $Q_B$ with
$x_e(e)=1$ and $x_e(e')=0, e'\in B\setminus\e.$ 
Then,
\begin{enumerate}
\item $x_e$ has minimal support in $\V;$
\item the collection of minimal support vectors in $\V$ spans $\V.$ Therefore if a vector is orthogonal to all minimal support vectors of $\V,$
then it is orthogonal to all vectors in $\V;$
\item if $\V$ is regular, then $Q_B$ has $0, \pm 1$ entries;
\item if $\V$ is regular,  then $Q_B$ is totally unimodular;
\item if $\V$ is regular, so is $\V^{\perp}.$
\end{enumerate}
\end{theorem}
\begin{proof}
1. Let $Q_B$ be an $r\times n$ matrix. Let $y$ be a vector in $\V$ such that $supp(y)$ is a proper subset of $supp(x_e).$
Now $\bf{y}= \bf{z}^TQ_B$ for some vector $z.$
Since $B$ is the set of  columns of an identity matrix, 
we must have $y(e')=z(e'), e'\in B.$ 
We have $supp(x_e)\cap B= \{e\}\supseteq supp(y)\cap B.$
So $z(e')=0, e'\ne e, e'\in B.$ It follows that if $y(e)\ne 0,$
then $y = z(e)x_e$ so that $supp(y)= supp(x)$ and if $y(e)= 0,$ then $y$  must be the zero vector.
Hence $x_e$ has minimal support in $\V.$

2. This follows since the rows of $Q_B$ form a basis for $\V$ and
have minimal support.
 
3. Since $\V$ is regular and $x_e$ has minimal support in $\V,$
we must have $x_e$ as a multiple of a $0, \pm 1$ vector.
But $x_e(e)=1.$
So $x_e$ must be a $0, \pm 1$ vector.

4. We will first show that every $r\times r$ submatrix has determinant $0,\pm 1.$
If the matrix is singular the determinant is zero. So let us suppose
that a submatrix $M$ is nonsingular with columns $B'.$
Consider the standard representative matrix $Q_{B'}.$
Let $N$ be the submatrix of $Q_{B'}$ with columns $B.$  
It is clear that $Q_B= MQ_{B'}$ and $NQ_B= Q_{B'},$
so that $MN=I,$ and $det(M)\times det(N)=1.$
Since $Q_B, Q_{B'}$ have $0, \pm 1$ entries, $det(M), det(N)$ are integers.
It follows that $det(M)$ is $ \pm 1.$

Next consider a $k\times k$ submatrix $P$ of $Q_B$ where $k\leq r.$
We can extend the column set of $Q_B$ corresponding to columns of $P$ using only columns
in $B$ so that we get a  nonsingular $r\times r$ submatrix $M'$.
Since the columns in $B$ are columns of the identity matrix,
we must have 
$\pm 1= det(M')=\pm det(P).$
We conclude that $ det(P)$ is $\pm 1.$

5. It is clear that $\bbmatrix{I&|&K}$ is a representative matrix
of $\V$ iff  $\bbmatrix{-K^T&|&I}$ is a representative matrix for $\V^{\perp}.$
Since, by Lemma \ref{base_circuit}, every minimal support vector is a multiple of a row of some
standard representative matrix, and,  we have shown above that, standard representative
matrices of $\V$ have $0,\pm 1$ entries, the result follows.
\begin{remark}
If the $j-$coboundary space is regular, it may not in general
be true that the $j-$coboundary matrix is totally unimodular.
For instance, the matrix\\
$$\bbmatrix{1&0&0&0&1\\1&-1&0&0&0\\0&1&1&0&0\\0&0&1&-1&0\\0&0&0&1&1}$$
is a submatrix of the  $2-$coboundary matrix of a $2-$complex embedded in $\mathbb{R}^3.$
But, as can be verified, it has determinant $2.$
For a discussion of these ideas, relating total unimodularity
to orientability, see \cite{dey}.
\end{remark}

\end{proof}
We have the following well known and fundamental results whose proofs, however, are simple.
\begin{theorem}
\label{perp_restrict_contract}
Let $\V$ be a vector space on a finite set $S$ and let $T\subseteq S.$  Then
\begin{enumerate}
\item $\V^{\perp\perp}=\V;$
\item $(\V\circ T)^{\perp}= \V^{\perp} \times T;$
\item $(\V\times T)^{\perp}=\V^{\perp} \circ T.$
\end{enumerate}
\end{theorem}

%


Let $\C$ be a $2-$complex embedded in $\mathbb{R}^3,$ $\T$ be a tetrahedron whose interior 
contains $\C,$ and let $\C_{\T}$ be defined as earlier
with the set of triangles $S:= S^{(2)}(\C)$ contained 
in the set of triangles  $S^{(2)}(\C_{\T}).$
We have 
\begin{lemma}
\label{cycles_cob_complex_rest}
Let $\V_{\T}$ be the space of $2-$coboundaries of $\C_{\T}.$
Then 
\begin{enumerate}
\item $\V_{\T}\circ S$ is the space of $2-$coboundaries of $\C;$
\item $\V_{\T}^{\perp}$ is the space of $2-$cycles of $\C_{\T}$
and $\V_{\T}^{\perp}\times S$ is the space of $2-$cycles of $\C.$
\end{enumerate}
\end{lemma}
\begin{proof}
1. We have,
$\V_{\T}$ as the row space of $A^{(2)}(\C_{\T}),$
and $\V_{\T}\circ S$ as the row space of the matrix obtained 
by deleting the columns of $A^{(2)}(\C_{\T})$
that are not in $S.$ But, except for zero rows, this matrix is the same as the coboundary matrix 
$A^{(2)}(\C),$
 of $\C.$
The result follows.

2. The spaces of $2-$coboundaries and of $2-$cycles of $\C_{\T}$ 
are complementary orthogonal by Theorem \ref{cycle_cob}.
Therefore $\V_{\T}^{\perp}$ is the space of $2-$cycles of $\C_{\T}$
and $(\V_{\T}\circ S)^{\perp}$ is the space of $2-$cycles of $\C.$
By Theorem \ref{perp_restrict_contract},
we therefore have that $\V_{\T}^{\perp}\times S$ is the space of $2-$cycles of $\C.$
\end{proof}
Let  $\V_{\T}$ be the space of $2-$coboundaries of $\C_{\T}.$
By Theorem \ref{minimal_2cycle},
we know that every minimal support $2-$cycle of  $\C_{\T}$ is a multiple of a 
$0,\pm 1$ vector which latter is the boundary vector of the union of a subset of regularly
oriented tetrahedra of $\C_{\T}.$
Let $\K_{\T}$ be the collection of boundary vectors of regions in $\mathbb{R}^3$ 
which are unions of tetrahedra of $\C_{\T}.$
By Theorem \ref{minimal_support}, such vectors span the space $\V_{\T}^{\perp}$ of $2-$cycles of $\C_{\T}.$
Let $\K\subseteq \K_{\T},$ be the collection of  
minimal support vectors with support contained in $S=S^{(2)}(\C).$ 
Clearly $\K$ spans 
the subspace of vectors of $\V_{\T}^{\perp}$  whose support is contained in $S,$
and therefore $\K\times S$
spans the  space $\V_{\T}^{\perp}\times S.$
Suppose a vector on $S$ is orthogonal to 
$\K\times S.$
Then it must also be orthogonal to $\V_{\T}^{\perp}\times S, $
i.e., (by Lemma \ref{cycles_cob_complex_rest})
to all $2-$cycles of $\C,$ i.e., must be a $2-$coboundary of 
$\C.$
Thus we have
\begin{lemma}
\label{cycle_orth}
Let the tetrahedra of $\C_{\T}$ be regularly oriented.
Let $\K$ be the collection of boundary vectors of regions in $\mathbb{R}^3$
which are unions of tetrahedra of $\C_{\T},$
and whose support is contained in the set of triangles $S=S^{(2)}(\C).$
If a vector on $S$ is orthogonal to all vectors of
$\K\times S,$ 
then it is a $2-$coboundary of
$\C.$
\end{lemma}
\section{Electrical $2-$networks}
\label{sec:electrical}
In this section we define electrical $2-$networks, prove a generalized Tellegen's Theorem
and describe procedures of solution for them.
\subsection{Flux and current adjusted mmf}
\label{def:flux_mmf}
Let $B$ and $H$ be the magnetic flux density
and magnetic intensity vectors defined on $\mathbb{R}^3.$
Let $\delta$ be an oriented triangle in $\mathbb{R}^3.$
Then the flux $\phi_{\delta}$ through the triangle, in the direction of the normal 
consistent with the orientation of the triangle (see Figure \ref{fig:triangle_orientation}), is equal to $ \int _{\delta}B.ds.$
The mmf $m_{\delta}'$ is the mmf associated with the  triangle in the direction of the normal.
This is the average value of $H\Delta l$ over the triangle, i.e., the value
of $\frac{\int _{\delta}H\Delta l.ds}{Area\  of\  \delta}.$ 
The triangles we work with are zero thickness idealizations of physical triangles with nonzero thickness $\Delta l.$
From Maxwell's equations, we have the relation          
$\oint_CH.dl = i,$ where $i$ is the current through an edge $e$ and $C$ is a contour around
the edge traversed in a direction consistent with the orientation of the edge.
(If $O$ is a point in $e,$ $A,B,$ are points encountered successively as we go around the contour,
the vector product $OA\times OB$ is in the direction of $e.$)
Suppose as in Figure \ref{fig:trianglesedge1}, we have triangles $\delta_1, \delta_2, \delta_3$
incident at and around an edge $e.$ Let the triangle orientation be such that in each case
it agrees with the orientation of $e$
and let a current $i_{\delta_j}$ circulate  around each triangle $\delta_j$ in the direction
of orientation of the triangle. 
Let us go around  $e,$ along a contour $C$ which intersects all the triangles incident 
at $e, $ in a direction consistent with the orientation of $e.$
The only nonzero contribution, to the integral $\oint_CH.dl,$ 
will be where the contour intersects the triangles. This is because
we are assuming that except for the triangles, the medium has infinite permeabilty and 
flux everywhere is finite. 
So the equation $\oint_CH.dl = i,$ reduces to
$m_{\delta_1}'+m_{\delta_2}'+m_{\delta_3}'=i_{\delta_1}+i_{\delta_2}+i_{\delta_3}=i.$
This equation can be rewritten as 
$m_{\delta_1}+m_{\delta_2}+m_{\delta_3}=0,$
where $m_{\delta_j}:= m_{\delta_j}'-i_{\delta_j}.$
We call the quantity $m_{\delta_j}:= m_{\delta_j}'-i_{\delta_j},$
the current adjusted mmf associated with ${\delta_j}.$


We state the following constraints in relation to the $3-$complex $\C_{\T}$
and the $2-$complex $\C$, both  embedded in $\mathbb{R}^3.$
\\

{\it{Constraint $K_1$(Generalized KVL)}

The net outward flux $\phi_S$ through $2-$cells of $\C$ which are boundaries of regions
of tetrahedra of $\C_{\T}$ is zero.
Equivalently, the flux vector $\phi_S$ is orthogonal to $2-$cycles of $\C$ which are boundaries of regions of tetrahedra of $\C_{\T}.$
\\

{\it{Constraint $K_2$}(Generalized KCL)}

If a set of triangles $S_e$ of $\C$ are incident at an edge $e$ and oriented to agree
with the orientation of $e,$
then $\sum_{\delta_i\in S_e}m_{\delta_i}=0,$ where $m_{\delta_i}$ is the `current adjusted
mmf' defined above associated with $\delta_i.$
Equivalently, the current adjusted
mmf vector $m_S$ is orthogonal to the rows of the coboundary matrix $A^{(2)}(\C)$
and is therefore a $2-$cycle of $\C.$ 
\\}

\subsection{Generalized Tellegen's Theorem}
\label{subsec:tellegen}
We are now in a position to define an electrical $2-$network $\N.$
\begin{definition}
\label{def:2network}
An electrical $2-$network $\N$ is a pair $(\C_S,\D_S),$
where\\
$\C_S$ is a $2-$complex embedded in $\mathbb{R}^3$ with $S$ as its $2-$cells,\\
$\D_S, $ the device characteristic, is a collection of ordered pairs $(\phi_S,m_S)$ of vectors on $S$ over $\mathbb{R}.$
A solution of $\N$ is an ordered pair $(\phi_S,m_S)$ 
satisfying the following constraints
\begin{enumerate}
\item $\phi_S $ satisfies Constraint $K_1;$
\item $m_S $ satisfies Constraint $K_2;$
\item  $(\phi_S,m_S)\in \D_S.$
\end{enumerate}

\end{definition}

We now have the following `generalized Tellegen's Theorem' for an electrical $2-$network $\N.$
\begin{theorem}
\label{gtellegen}
Let $\N:= (\C_S,\D_S),$
be an electrical $2-$network.
Let $\V^{\phi}_S, \V^{m}_S$ be the spaces of vectors which satisfy 
constraints $K_1,K_2$ respectively for $\N.$
Then $\V^{\phi}_S, \V^{m}_S$ are complementary orthogonal, being respectively
the $2-$coboundary and $2-$cycle spaces of $\C.$
\end{theorem}
\begin{proof}
Constraint $K_1$ states that the flux vector $\phi_S$ is orthogonal to $2-$cycles of $\C$ which are boundaries of regions of tetrahedra of $\C_{\T}.$
This, by Lemma \ref{cycle_orth}, is equivalent to saying that it is orthogonal to all  $2-$cycles of $\C$ and therefore is a $2-$coboundary of $\C.$
Thus $\V^{\phi}_S$ is the space of $2-$coboundaries of $\C.$

Constraint $K_2$ states that the current adjusted mms vector $m_S$ is a $2-$cycle of $\C.$
Thus $\V^{m}_S$ is the space of $2-$cycles of $\C.$

Since by Theorem \ref{cycle_cob}, the spaces of $2-$cycles and of $2-$coboundaries of $\C$
are complementary orthogonal, the result follows.

\end{proof}
Let $A:= A^{(2)}(\C),$
be the $2-$coboundary matrix of $\C.$
Using Theorem \ref{gtellegen}, we can restate the constraints of the network $\N$ as follows.
\begin{equation}
\label{network_constraints_2}
\begin{matrix}
A\ {\bf m_S}&=&{\bf 0}.\\
A^T\ {\bf y}&=&{\bf\phi_S}\\
(\phi_S,m_S)&\in &\D_S.
\end{matrix}
\end{equation}
In the special case where the flux is linearly related to mmf, we may write
\begin{equation}
G\ {\bf\phi_S}= [{\bf m_S}+{\bf I_S}],
\end{equation}
where $G$ is a positive diagonal matrix with its $(j,j)$ entry being the 
reluctance of the triangle $\delta_j,$ $m_S$ is the current adjusted mmf vector and
$\phi_S$ is the flux vector associated with $S=S^{(2)}(\C).$
\subsection{Solution of an electrical $2-$network}
\label{subsec:solution}
We now discuss the procedure for solving Equation \ref{network_constraints_2},
when the device characteristic is linear.

Let $\hat{A}$ be composed of a maximal linearly independent set of rows of the $2-$coboundary matrix $A$ of $\C.$
We have, writing $\hat{A}{\bf m_S}={\bf 0}$ as $\hat{A} [{\bf m_S}+{\bf I_S}]=\hat{A}{\bf I_S}$
and $A^T\ {\bf y}={\bf\phi_S}$ as $\hat{A}^T\ {\bf z},$
\begin{equation}
\label{network_constraints_3}
\begin{matrix}
 \hat{A}\ [{\bf m_S}+{\bf I_S}]&=&\hat{A}\ {\bf I_S},i.e.,\\
\hat{A}\ G\ {\bf\phi_S}&=&\hat{A}\ {\bf I_S}, i.e., \\
\end{matrix}
\end{equation}
\begin{equation}
\label{network_constraints_3}
\begin{matrix}
\hat{A}\ G\ \hat{A}^T\ {\bf z}&=&\hat{A}\ {\bf I_S}. 
\end{matrix}
\end{equation}


When $G$ is positive diagonal, since the rows of $\hat{A}$ are linearly
independent, the matrix $\hat{A}G\hat{A}^T$ is positive definite and so invertible
so that a unique solution for $z$ and therefore also for the variables $\phi_S,m_S,$ is guaranteed. If $\C$ is large the matrix $\hat{A}G\hat{A}^T$ would be sparse so that existing efficient techniques for solution of
such equations can be adopted. As opposed to conventional graph based electrical
networks, the main extra computational effort is in computing
$\hat{A}$ from $A.$ In the case of graph based circuits this computation 
is trivial since we simply have to drop one row per connected component of 
the graph. Here it appears as though we require linear algebraic computations
which though efficient are not trivial.
There are other advantages to dealing with graph based circuits - the equations
can be `preprocessed' in near linear time to make them more suitable 
for, say, parallel processing. Such convenient structures are obtained, for instance, through `multiport decomposition' and `topological transformation'
(\cite{book}). The algorithms required are graph based.
It is not immediately clear if anything similar can be done in the present situation.

We will show, in subsequent pages, that the present network can actually be reduced
to a graph based `dual network'. To appreciate what would be the procedure 
after such conversion let us discuss a cycle-based rather than coboundary-based
(as above) equations for the network.

\subsection{Cycle based equations for the network $\N$}
\label{subsec:cycleeqns}
Let $P$ be a representative matrix of the $2-$cycle space of $\C.$
Using Theorem \ref{gtellegen}, we can restate the constraints of the network $\N$ as follows.
\begin{equation}
\label{network_constraints_4}
\begin{matrix}
P\ {\bf \phi_S}&=&{\bf 0}.\\
P^T\ {\bf q}&=&{\bf m_S}\\
(\phi_S,m_S)&\in &\D_S.
\end{matrix}
\end{equation}
In the special case where the flux is linearly related to mmf, we may write
\begin{equation}
{\bf\phi_S}= R\ [{\bf m_S}+{\bf I_S}].
\end{equation}
Let us, for  handling more general but analogous situations,
replace this equation by
\begin{equation}
{\bf\phi_S}+{\bf E_S}= R\ [{\bf m_S}+{\bf I_S}],
\end{equation}
where ${\bf E_S}$ is an additional source term.
We have, 
\begin{equation}
\begin{matrix}
 P\ {\bf \phi_S}&=&{\bf 0},i.e.,\\
 P\ [{\bf \phi_S}+{\bf E_S}]&=&P\ {\bf  E_S},i.e.,\\
P\ R\ [ {\bf m_S}+{\bf I_S}]&=&P\ {\bf  E_S}, i.e., \\
P\ R\ {\bf m_S}&=&- P\ R\ {\bf I_S}+ P\ {\bf  E_S}, i.e.,
\end{matrix}
\end{equation}

\begin{equation}
\label{network_constraints_5}
\begin{matrix}
P\ R\ P^T\ {\bf q}&=&- P\ R\ {\bf I_S} + P\ {\bf  E_S}.
\end{matrix}
\end{equation}
When $R$ is positive diagonal, since the rows of $P$ are linearly
independent, the matrix $P R P^T$ is positive definite and so invertible
so that a unique solution is guaranteed. 
We will show that the matrix $P$ can be chosen as the reduced incidence matrix
(dropping one row per connected component from the incidence matrix) of a graph
which is the `matroid dual' of $\C.$
Thus Equation \ref{network_constraints_4} will look exactly like the constraints
of a graph based network.
We will build this latter graph in linear time from $\C .$
\section{Triangle adjacency graph $tag(\C)$ and the graph $\G_{tag(\C)}$ }
\label{sec:triangle}
\begin{figure}
\centering
 \includegraphics[width=1.5in]{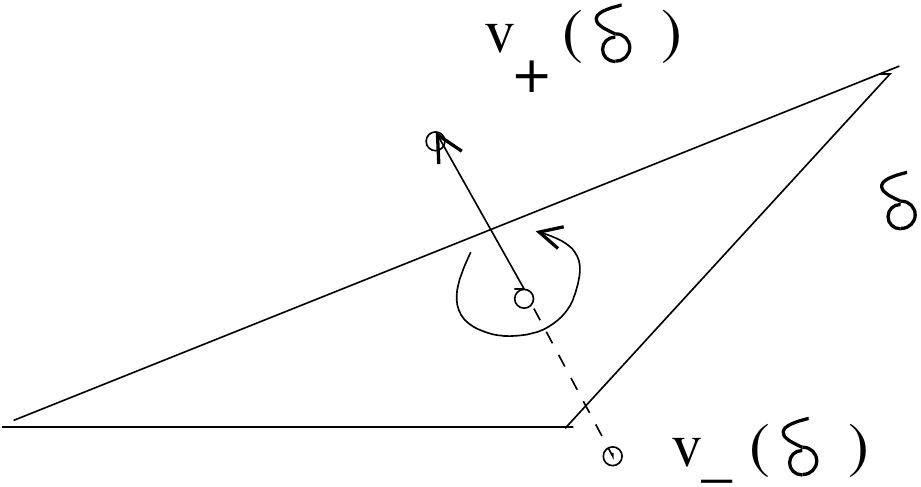}
 \caption{$v_+(\delta)$ and $v_-(\delta)$ for triangle $\delta $}
 \label{fig:triangle_orientation2}
\end{figure}

\begin{figure}
\centering
 \includegraphics[width=7in]{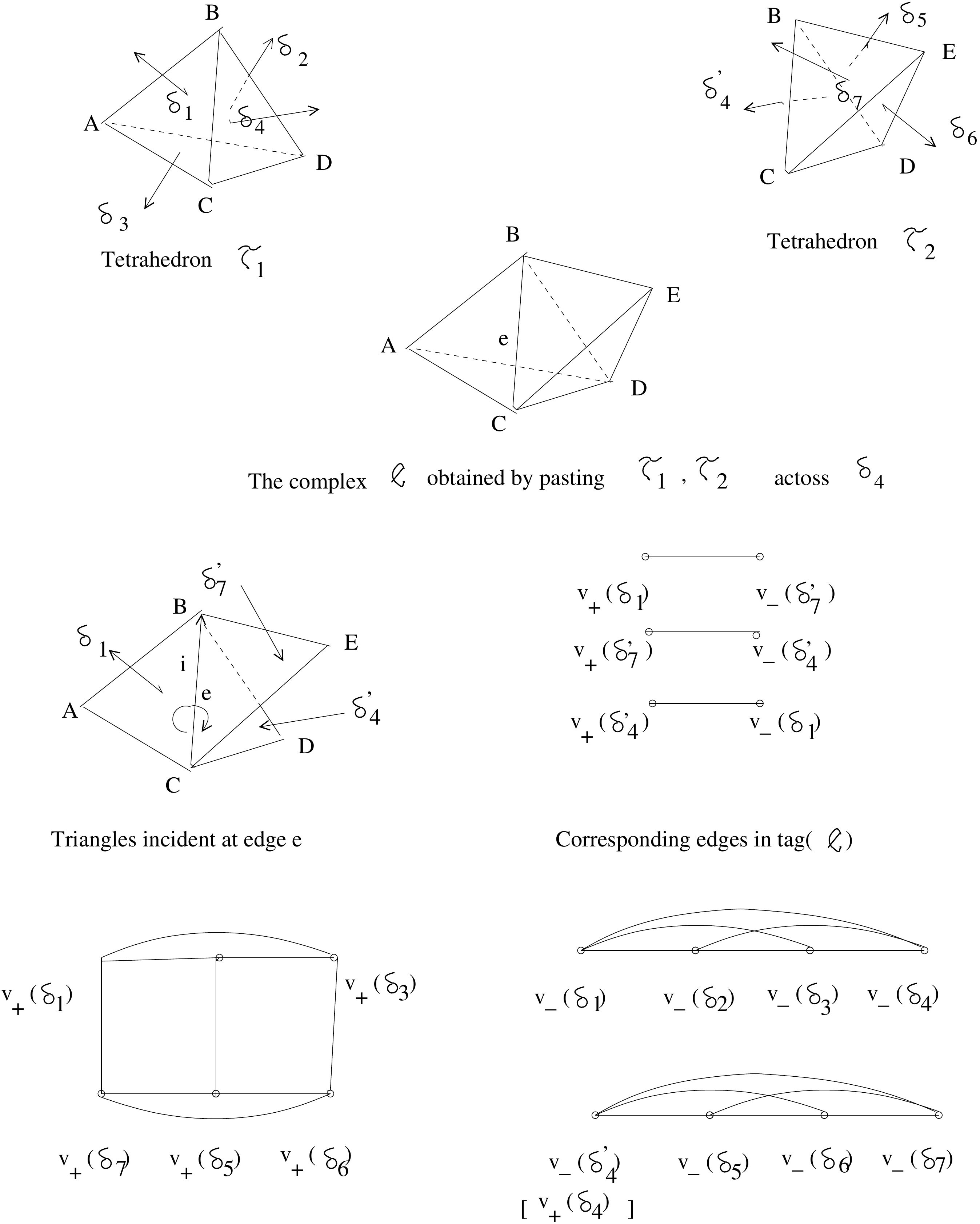}
 \caption{The graph $tag($\C$)$  for a complex $\C$}
 \label{fig:tag}
\end{figure}
The triangle adjacency graph $tag(\C)$ of $\C$ is constructed from the embedding information
about $\C$ in  $\mathbb{R}^3$ as to how triangles lie around an edge they are incident on.
We show later, that for the purpose of writing minimal equations for the network, 
we only have to find the connected components of $tag(\C).$ 

Let $\C$ be a $2-$complex embedded in $\mathbb{R}^3.$

We assign each oriented triangle $\delta:= abc,$ a `positive' node $v_+(\delta)$ and a `negative'
node  $v_-(\delta).$
Physically, $v_+(\delta)$  is placed at a distance $\epsilon$  from the centroid of the triangle moving in the  direction of $ab \times bc,$ (i.e.,
normal to the triangle consistent with the orientation of the triangle) and $v_-(\delta)$  is placed at a distance $\epsilon$  from the centroid of the triangle 
in the opposite direction (see Figure \ref{fig:triangle_orientation2}). The value of $\epsilon$ is sufficiently small so that
the line segment  $(v_+(\delta),v_-(\delta))$
does not intersect any other triangle of $\C.$
Note that if $\delta$ is given the opposite orientation and denoted $\delta ',$ then
$v_+(\delta')=v_-(\delta)$ and $v_-(\delta')=v_+(\delta).$
\begin{definition}
To build $tag(\C),$ we  start with the vertex set $$V(tag(\C)):= \{v_+(\delta_i),\delta_i\ \mbox{a triangle in}\ \C\}\bigcup \{v_-(\delta_i),\delta_i\ \mbox{a triangle in}\ \C\}.$$
Suppose triangles $\delta_1,\cdots ,\delta_k$
are incident at edge $e$ and are oriented to agree with the orientation of $e.$ Further, suppose when we rotate around $e$ in a direction consistent with the 
orientation of $e,$ we encounter them successively as $\delta_1,\delta _2,\cdots ,\delta_k, \delta _1$
(see Figure \ref{fig:trianglesedge1} for the case $k=3$).
Then in $tag(\C),$ we construct the undirected edges 
$$(v_+(\delta_1),v_-(\delta_2)),(v_+(\delta_2),v_-(\delta_3)), \cdots ,(v_+(\delta_{k-1}),v_-(\delta_k)),(v_+(\delta_k),v_-(\delta_1)).$$ 
This is done for every edge $e$ of $\C$ (see Figure \ref{fig:tag}). The resulting set of edges is the edge set of
$tag(\C).$
\end{definition}
\begin{remark}
We note that in the above definition, `$v_+(\delta_i),v_-(\delta_i)$'
are used to denote both  points in $\mathbb{R}^3$ and the corresponding vertices
in $tag(\C).$ It would be clear from the context which entity is being referred to.
\end{remark}
In Figure \ref{fig:tag}, we have constructed the graph $tag(\C)$ 
for a complex $\C$ which is obtained by pasting  two tetrahedra $\tau_1,\tau_2$ across the face $\delta_4.$
Thus the complex has $3-$cells $\tau_1,\tau_2, $  $2-$cells $\delta_1, \delta_2,\delta_3,\delta_4,\delta_5,\delta_6,\delta_7$  and $1-$cells and $0-$cells, the edges and vertices of these triangles. 
The tetrahedra $\tau_1,\tau_2 $ are regularly oriented 
(i.e., so that in the boundary the triangles appear with 
an orientation consistent with the outward normal as in Figure \ref{fig:regular}).
The triangle $\delta_4$ in tetrahedron $\tau_1$ appears as $\delta'_4$ with opposite orientation 
in tetrahedron $\tau_2.$ 
Observe that this graph has three connected components corresponding to
the `outer surface' of the complex composed of the triangles
$\delta _1,\delta _2,\delta _3,\delta _5,\delta _6,\delta _7$
and the `inner surfaces' of the tetrahedra $\tau_1,\tau_2$
composed respectively of the triangles $\delta _1,\delta _2,\delta _3,\delta _4$ and $\delta_4',\delta _5,\delta _6,\delta _7.$ 
Our regular orientation of the tetrahedra has led to these surfaces 
being represented by $v_+$ vertices for the outer surface 
and $v_-$ vertices for the inner surfaces (noting that $\delta_4=- \delta_4'$).

To build the graph $\G_{comptag(\C)}$, we first find the connected components of
the graph $tag(\C).$ Let $V_1, \cdots , V_k$ be the vertex sets of these
components. We take the vertex set $V(\G_{comptag(\C)})$ to be the 
set $\{V_1, \cdots , V_k\}.$ Let $\delta $ be a triangle of $\C$
and let $v_+(\delta)\in V_i, v_-(\delta)\in V_j.$  Then in $\G_{comptag(\C)}$
we introduce a directed edge $e(\delta)$ from $V_i$ to $V_j.$ 
(If $V_i=V_j$ for $\delta$ then $e(\delta)$
would be a self loop.)
This is done for each triangle $\delta.$ 
Figure \ref{fig:gtagc} shows the graph $\G_{comptag(\C)}$ for the complex 
in Figure \ref{fig:tag}. This has been constructed by building
the three supernodes 
$$V_1:= \{v_+(\delta_1),v_+(\delta_2),v_+(\delta_3),v_+(\delta_5),v_+(\delta_6),v_+(\delta_7)\},
V_2:= \{v_-(\delta_1),v_-(\delta_2),v_-(\delta_3), v_-(\delta_4)\},$$
$$V_3:= \{v_-(\delta_4'),v_-(\delta_5),v_-(\delta_6),v_-(\delta_7)\}\equiv \{v_+(\delta_4),v_-(\delta_5),v_-(\delta_6),v_-(\delta_7)\}$$
and introducing an edge for each $\delta$ from the super node in which its
$v_+$ vertex lies to the supernode in which its $v_-$ vertex lies.

It is clear that when $\C$ is removed from $\T,$ $\T\setminus\C$
could split into connected regions
 (see beginning of Section \ref{sec:2complex} for a description of $\C,\T, \C_{\T}$ etc.).
We now have the following intuitively obvious result.
\begin{lemma}
\label{lem:tagpathregionpath}
If $v_{\pm}(\delta_1)$ and $v_{\pm}(\delta_2)$ are path connected in
$tag(\C),$ then they are also path connected in $\T\setminus\C.$
\end{lemma}
\begin{proof}
It is sufficient to consider the case where the path length in $tag(\C)$ is of length $1,$  i.e., when 
$\delta_1,\delta_2$ share a common edge $e$ in $\C.$
We may assume that $e$ appears with a positive sign in the boundary of both 
$\delta_1$ and $\delta_2.$
This means in $\mathbb{R}^3,$ as we rotate about $e$ consistent with its orientation,
we encounter $\delta_2$ immediately after $\delta_1$ (see Figure \ref{fig:trianglesedge1}).
In $\C_{\T},$ there would be tetrahedra  $\tau_1, \tau_2$ such that
 $v_{+}(\delta_1)$  belongs to tetrahedron $\tau_1$ and $v_{-}(\delta_2)$  belongs to tetrahedron $\tau_2.$
In $\C_{\T},$ there would be tetrahedra $\tau_1', \cdots , \tau_k'$
all having $e$ as an edge,
such that $\tau_1,\tau_1'$ share a common triangle,  $ \tau_i',\tau_{i+1}', i= 1, \cdots k-1,$
share a common triangle,  $\tau_k',\tau_2$
share a common triangle with none of these common triangles  belonging to $\C.$
Draw an arc of radius $\epsilon$ from a point $x_1$ in $\tau_1$ to a point 
$x_2$ in $\tau_2.$ 
Choose $\epsilon$ to be sufficiently small so that it lies entirely in the union
of the tetrahedra $\tau_1,\tau_1', \cdots , \tau_k',\tau_2.$
The path $v_{+}(\delta_1),x_1$ lies in $\tau_1,$ the path $x_2,v_{-}(\delta_2)$
lies in $\tau_2,$ and the $\epsilon -$arc from $x_1$ to $x_2$ lies in the union
of the tetrahedra $\tau_1,\tau_1', \cdots , \tau_k',\tau_2.$
This path will continue to exist even if $\C$ is deleted from $\T.$
Observe that this argument is valid in the simpler case where $\tau_1,\tau_2$ share a common triangle.
The result follows.   
\end{proof}
We will call the algorithm for finding the  connected component of $tag(\C),$
the `sliding algorithm', because if we translate the traversal of nodes of $tag(\C)$
to a `physical' traversal of the triangles of $\C,$ it appears as though we are sliding on the surface
of triangles of $\C.$
  
It is clear that the construction of $tag(\C)$ from $\C$ is linear time on the size of $\C.$
Since finding connected component of a graph is linear time on the size of the graph,
it is clear that the construction of  $\G_{comptag(\C)}$
from $\C$ is also linear time on the size of $\C.$

\begin{figure}
\centering
 \includegraphics[width=7in]{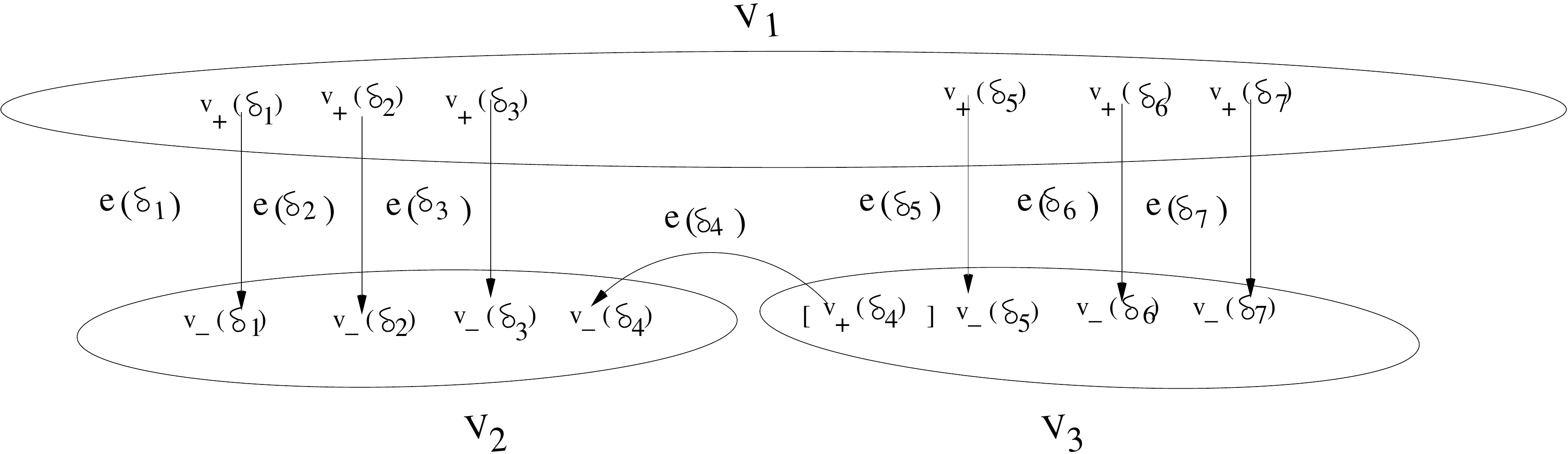}
 \caption{The graph $\G_{comptag(\C)}$  for the complex $\C$}
 \label{fig:gtagc}
\end{figure}

\begin{definition}
Let $\G$ be a directed graph. The incidence matrix $A(\G)$ of $\G$ is defined as follows.
$A(\G)$ has rows corresponding to nodes of $\G$ and columns corresponding
to edges of $\G.$
The entry $A(\G)_{i,j}$ is zero if edge $j$ is not incident on node $i,$
is $+1$ if the edge $j$ is incident on node $i$ but directed away
and is $-1$ if the edge $j$ is incident on node $i$ but directed inward.
If an edge has only one end point, then the corresponding column is a zero column.

\end{definition}
We have the following well known lemma whose proof is routine (see \cite{book},
for instance).
\begin{lemma}
\label{lem:incidencerank}
Let $\G$ be a directed graph. Then
\begin{enumerate}
\item The sum of the rows of the incidence matrix $A(\G)$ of $\G$ is zero;
\item If $\G$ is connected, then the matrix obtained by deleting any row of
$A(\G)$ has linearly independent rows.
\end{enumerate}
\end{lemma}

The incidence matrix $A(\G_{comptag(\C)})$  
has rows corresponding to nodes of $\G_{comptag(\C)},$ i.e.,
to vertex sets of connected components of  $tag(\C),$ and columns corresponding 
to edges of $\G_{comptag(\C)},$ i.e.,
triangles of $\C.$
Let us denote by $V_i$ both a vertex of $\G_{comptag(\C)}$
as well as the corresponding vertex set of the connected component of $tag(\C).$ 
\\

Thus the entry $A(\G_{comptag(\C)})_{V_i,\delta _j}$ 
is zero if neither $v_+(\delta_j)$ nor $ v_-(\delta_j)$ belongs to $V_i$ in $tag(\C),$
is zero if both of  $v_+(\delta_j), v_-(\delta_j)$ belong to $V_i$ in $tag(\C),$
is $+1$ if  $v_+(\delta_j)$  belongs to $V_i$ in $tag(\C),$
is $-1$ if $v_-(\delta_j)$  belongs to $V_i$ in $tag(\C).$
\\

The incidence matrix $A(\G_{comptag(\C)})$  of the graph $\G_{comptag(\C)}$
in Figure \ref{fig:gtagc},
constructed from the triangle adjacency graph $tag(\C),$
associated with the complex $\C$ in Figure \ref{fig:tag}
is given below.

$$
\begin{matrix}
\ \   &\delta_1&\delta_2&\delta_3&\delta_4&\delta_5&\delta_6&\delta_7\\  
V_1   &+1      &+1      &+1      &0       &+1      &+1      &+1\\
V_2&          -1&-1&-1&-1&0&0&0\\
V_3&          0&0&0&+1&-1&-1&-1
\end{matrix}
.$$
We now have the following simple lemma
\begin{lemma}
\label{lem:incidence}
The rows of $A(\G_{comptag(\C)})$ are $2-$cycles of $\C.$
\end{lemma}
\begin{proof}
We will show that the rows of $A(\G_{comptag(\C)})$ are orthogonal to
the rows of $A^{(2)}(\C),$ the $2-$ coboundary matrix of $\C.$
Consider the `$e$' row corresponding to edge $e$ of $\C.$
Let $\delta_1, \cdots , \delta_k$ be the triangles incident at $e.$
Let us, without loss of generality, assume that all these triangles 
have been oriented to agree with the orientation 
of $e.$
Thus the row $e$ in $A^{(2)}(\C)$ has only $0$ and $+1$ entries.
Let us assume further that $\delta_1, \cdots , \delta_k, \delta_1$ is the order
in which we encounter the triangles as we rotate around $e$
in $\mathbb{R}^3$ consistent with the orientation of $e.$
(See Figure \ref{fig:trianglesedge1} for the case $k=3$). 

Let us call the ordered pairs $(\delta _1, \delta _2), \cdots ,
(\delta _{i}, \delta _{i+1}), \cdots ,(\delta_k, \delta_1)$
ordered pairs of adjacent triangles at $e.$
For each ordered pair $(\delta _{i}, \delta _{i+1})$ we will have an edge in $tag(\C)$
between $v_+(\delta_i)$ and $v_-(\delta_{i+1}).$ 
We may  ignore all the triangles $\delta$ incident at $e$ 
 such that in $tag(\C)$ 
both $v_+(\delta)$ and $v_-(\delta)$ belong to the same component,
since these contribute zero entries to the rows of $A(\G_{comptag(\C)}).$
After this,  we can see that of $\delta_1, \cdots , \delta_k$
we must have the same number of $v_+(\delta)$ and $v_-(\delta)$ belonging to a given component.
Thus in the row of $A(\G_{comptag(\C)})$ corresponding to this 
component we have equal number of $+1$s and $-1$s in columns corresponding to triangles incident at $e$ in $\C.$ 
But the row $e$ in $A^{(2)}(\C)$ has only $+1$ entries corresponding to triangles incident at it.
Therefore,  the two vectors are orthogonal.
Thus every row of $A(\G_{comptag(\C)})$ is orthogonal
to every row of $A^{(2)}(\C),$ $i.e.,$ rows of $A(\G_{comptag(\C)})$ are $2-$cycles of $\C.$

%
%
\end{proof}
We show later that the rows of $A(\G_{comptag(\C)})$ actually generate the 
$2-$cycle space of $\C.$

\section{Cell dual $\G_{\T}$ of $\C_{\T}$ and the region graph $\G_{region(\C)}$ of $\C$}
\label{sec:celldual}
The cell dual of the $n-$cells of an $n-$complex is a simple concept  for studying the interrelationship
of the $n-$cells. In practice, this idea is often used to solve electromagnetic field problems,
for instance, in the case of magnetic circuits.
Our interest is in the cell dual of the $3-$complex $\C_{\T}$ embedded in $\mathbb{R}^3.$
Section \ref{sec:2complex} has a  description of $\C_{\T}.$ 
We remind the reader that the set $S$ of triangles of $\C$ is contained in the set $S^{(2)}(\C_{\T})$
of triangles of $\C_{\T}.$ Further 
the  triangles of $\C$ lie in the interior of the tetrahedron $\T$
so that 
none of these triangles intersect the subset 
of triangles of $S^{(2)}(\C_{\T})$ which lie in the boundary of $\T.$
The notions $v_+(\delta),v_-(\delta)$ are defined in the beginning of Section \ref{sec:triangle}.
We show in this section that the graph $\G_{region\C},$ defined below, which captures the interrelationship between 
connected regions of $\T\setminus\C,$ can be obtained from the cell dual of $\C_{\T}$ by  contracting
edges which correspond to triangles of $\C_{\T}$ which are not in $\C.$ 
\begin{definition}
Let the triangles in $S$ have the same orientation in $\C$ and $\C_{\T}.$
The cell dual $\G_{\T}$ of $\C_{\T}$ is constructed as follows. Put a node $v_{ext}$ in the region 
external to $\T.$ In the interior of each tetrahedron $\tau$ of $\C_{\T}$ put a node $v_{\tau }.$
For each triangle $\delta$ of $\C_{\T}$ which has $v_+(\delta) \in \tau_i$ and $v_-(\delta) \in \tau_j$
introduce a directed edge $e_{\delta}$ from $v_{\tau _i}$ to 
$v_{\tau _j}.$ If $\delta$ lies in the boundary of $\T$ and in the boundary of the tetrahedron $\tau,$ add an edge between $v_{ext}$ 
and $v_{\tau },$
with direction  from $v_+(\delta) $ to $v_-(\delta).$
\end{definition}
The following lemma is immediate  from the definition of the cell dual.

\begin{lemma}
\label{lem:connectcelldual}
$\G_{\T}$ is connected.
\end{lemma}
We now have the following lemma.
\begin{lemma}
\label{lem:incidence2}
Let $A({\G_{\T}})$ be the incidence matrix of  $\G_{\T}.$ Let $A_{red}$ be the submatrix 
of $A({\G_{\T}})$ obtained by deleting the row corresponding to $v_{ext}.$
Then 
\begin{enumerate}
\item $A_{red}$ is the transpose of the $3-$coboundary matrix of $\C_{\T};$
\item in $A({\G_{\T}}),$ the row corresponding to $v_{ext}$ is the negative of the sum
of all the other rows;
\item the row space of  $A({\G_{\T}})$ = row space of $A_{red}$ = the $2-$cycle space of $\C_{\T}.$
\end{enumerate}
\end{lemma}
\begin{proof}
Parts (1) and (2) are immediate from the definitions of $\G_{\T}, A({\G_{\T}}), 
A_{red} $ and $3-$coboundary matrices.

Part (3) follows from Lemma \ref{lem:connectcelldual}, Lemma \ref{lem:incidencerank}
and Theorem \ref{Hurewicz2}.

\end{proof}

\begin{definition}
The region graph $\G_{region(\C)}$ of $\C$ has nodes corresponding to connected regions 
$\Reg_1, \cdots, \Reg_k$ of $\T \setminus\C$ and edges corresponding to triangles of $\C.$
It is built as follows.  
Put a node $v_i$ in the region $\Reg_i$ for  $i=1, \cdots ,k.$
For each triangle $\delta $ of $\C$ introduce a directed edge $e_{\delta}$ from $v_i$ to $v_j$
if $v_+(\delta) \in \Reg_i$ and $v_-(\delta) \in \Reg_j.$
\end{definition}
\begin{definition}
\label{def:graphcrossdot}
Let $\G$ be a graph on vertex set $V$ and edge set $E.$ Let $T\subseteq E.$
The graph $\G\circ T,$ called restriction of $\G$ to $ T,$ is obtained from $\G$ by deleting edges in $E\setminus T$ and also 
deleting any isolated nodes formed in the process.\\ 
The graph $\G\times T,$ called contraction of $\G$ to $ T,$ is obtained
by first fusing the end vertices of edges in  $E\setminus T$ and then deleting them. We
 say, in this case, that the edges of $E\setminus T$ are contracted.\\
If $T_2\subseteq T,$ we denote $(\G\times T)\times T_2, (\G\times T)\circ T_2,(\G\circ T)\times T_2,(\G\circ T)\circ T_2,$ respectively by\\ $\G\times T\times T_2, \G\times T\circ T_2,\G\circ T\times T_2,\G\circ T\circ T_2.$\\
A graph of the form $\G\circ T\times T_2$ or of the form $\G\times T\circ T_2$ is called a minor of $\G.$
\end{definition}
\begin{definition}
Let $\G$ be a graph on vertex set $V$ and edge set $E.$
$\V(\G)$ denotes the $1-$coboundary space on $E,$ i.e., the row space of the incidence
 matrix of $\G.$ 
\end{definition}

In the literature,  the notation  `$\circ$', `$\times$' are used for operations on graphs as above, 
for operations on vector spaces as in Definition \ref{def:minor} and for operations on matroids 
as later in Definition \ref{def:matroiddotcross}.
This is partly because the operations are related.
The context would make clear as to which operation is intended.
This is so in the following results which are  well known (\cite{tutte},\cite{book}).
(For better readability set difference is denoted $A-B$ rather than $A\setminus B.$)
\begin{theorem}
Let $\G$ be a graph on vertex set $V$ and edge set $E.$ Let $T_2\subseteq T.$ 
We have
\begin{enumerate}
\item $\G\times T\circ T= \G\times T,$
\item $\G\circ T\times T= \G\circ T,$
\item $\G\times T\times T_2= \G\times T_2,$ 
\item $\G\times T\circ T_2= \G\circ (E-( T- T_2))\times T_2,$
\item $\G\circ T\circ T_2= \G\circ T_2.$
\end{enumerate}
\end{theorem}
\begin{theorem}
\label{thm:graphcontrest}
Let $\G$ be a graph on vertex set $V$ and edge set $E $ and let $T\subseteq E.$ 
Let $\V(\G')$ denote the row space of $A(\G').$
\begin{enumerate}
\item $\V(\G\times T)= (\V(\G))\times T,$
\item $\V(\G\circ T)= (\V(\G))\circ T.$
\end{enumerate}
\end{theorem}

We now have the following simple lemma.
\begin{lemma}
\label{lem:cycleregion}
Let $S$ denote the set of edges of $ \G_{\T}$ corresponding to triangles
in the complex $\C.$
We have
\begin{enumerate}
\item $\G_{region(\C)} = \G_{\T}\times S;$
\item $\G_{region(\C)}$  is connected.
\item $\V(\G_{region(\C)})=\V \times S$, where $\V$ is the $2-$cycle space of $\C_{\T}$ and hence $\V(\G_{region(\C)})$ is the $2-$cycle space of $\C.$
\end{enumerate}
\end{lemma}
\begin{proof}

1.  Two tetrahedra $\tau_{0},\tau_{end}$ of $\C_{\T}$ belong to a connected region of
$\T\setminus\C,$ iff we can find a sequence $\tau_{0},\tau_1, \cdots , \tau_n,\tau_{end},$ 
such  that each tetrahedron, other than $\tau_{end},$  has a common triangle 
with the next tetrahedron, with the additional condition that this common triangle
does not belong to $S.$ In $\G_{\T},$ these correspond to nodes $v_{\tau_{0}},v_{\tau_1}, \cdots , v_{\tau_n},v_{\tau_{end}},$ with successive nodes being joined by an 
edge $e_{\delta}$ with $\delta \notin S.$ 
If we contract all the edges of $\G_{\T}$ corresponding to triangles of $\C_{\T}$ that lie in a connected region of $\T\setminus \C,$ this would result in all the nodes of
$\G_{\T}$ corresponding to  tetrahedra 
in the connected region  
being fused to a single node. When this is done repeatedly for all the connected
regions of $\T\setminus \C,$ we would be left with a graph which has one node $v_{\Reg _i}$
coresponding to each connected region $\Reg_i$ of $\T\setminus \C$ and edges between these
nodes which correspond to those triangles in $S$ which lie between the connected regions.
An edge $e_{\delta}$ would be directed from $v_{\Reg_i}$ to 
$v_{\Reg_j},$
if $v_+(\delta)$ lies in $\Reg_i $ and
if $v_-(\delta)$ lies in $\Reg_j .$ 
Since the operations we have performed are contractions of edges in $E(\G_{\T})- S,$
in the graph $\G_{\T},$ and this resulting graph is $\G_{region(\C)},$ the lemma follows.

2. From Lemma \ref{lem:connectcelldual} we know that $ \G_{\T}$ is connected.
Contraction of a connected graph results in a connected graph. The result follows.

3. By Lemma \ref{lem:incidence2}, the row space of $A(\G_{\T})$ is the $2-$cycle space of
$\C_{\T}.$
By Lemma \ref{cycles_cob_complex_rest}, if $\V$ is the $2-$cycle space of
$\C_{\T},$ then $\V\times S$ is the  $2-$cycle space of
$\C.$ By Theorem \ref{thm:graphcontrest}, if $\hat{\V} $ is the row space of the incidence matrix
of $A(\G_{\T}),$ then $\hat{\V}\times S$ is the row space of the incidence matrix 
of $A(\G_{\T}\times S).$  
By part (1) of the present lemma, $\G_{region(\C)} = \G_{\T}\times S.$
The result follows.

\end{proof}
\section{$\G_{region(\C)} = \G_{comptag(\C)}$ for connected $\C$}
\label{sec:region=tag}
We show in this section that $\G_{region(\C)},\ \G_{comptag(\C)}, $ are identical
if $\C$ is connected (see Definition \ref{def:connect} for definition of connectedness of complexes).
\\

We have seen that the row space
of the incidence matrix of $\G_{region(\C)}$ is the cycle space of $\C.$
Hence, by the discussion in Subsection \ref{subsec:cycleeqns}, if $\G_{region(\C)}$ can be constructed easily,
we have an efficient way of writing equations for the solution
of the electrical $2-$network based on $\C.$
The definition of $\G_{region(\C)}$ requires the notion of path connectedness
in $\T\setminus \C.$ In practice, determining whether two points in $\mathbb{R}^3$
can be connected by a path that avoids specified obstacles is cumbersome.
On the other hand, as we have seen, building $tag(\C)$ and $\G_{comptag(\C)}$ 
from $\C$ embedded in $\mathbb{R}^3 $ is linear time on the size of $\C.$
When $\C$ is disconnected we merely have to repeat the procedure 
for each connected component of $\C.$
Thus, the fact that $\G_{region(\C)} = \G_{comptag(\C)}$ for connected $\C,$
is very useful for our purposes.

We will prove the main theorem of this section through a series of lemmas.
\begin{theorem}
If $\C$ is connected, then $\G_{region(\C)} = \G_{comptag(\C)}.$
\end{theorem}
Let the edges of  $\G_{region(\C)},\   \G_{comptag(\C)}$
be named according to the triangles of $\C,$ to which
they correspond.
\begin{lemma}
The row space of $A(\G_{comptag(\C)})$ is contained in the row space
of $A(\G_{region(\C)}).$
\begin{proof}
By Lemma \ref{lem:incidence}, rows of $A(\G_{comptag(\C)})$ are $2-$cycles of $\C,$
and by  Lemma \ref{lem:cycleregion}, row space of $A(\G_{region(\C)})$ is the 
space of $2-$cycles of $\C.$
The result follows.
\end{proof}
\end{lemma}

\begin{lemma}
\label{lem:twographequality}


1. For each row $i$ of $A(\G_{comptag(\C)})$ there exists a row $i'$ of 
 $A(\G_{region(\C)})$ such that whenever  
$A(\G_{comptag(\C)})(i,j)$ is nonzero, we have $A(\G_{comptag(\C)}(i,j)=A(\G_{region(\C)})(i',j)$ and further, if $\delta$ is a selfloop incident  at $i$ in $\G_{comptag(\C)}$
 it is a selfloop incident at $i'$  in $\G_{region(\C)}.$

2. When $\C$ is connected, this map $i \rightarrow i'$ is one to one onto
and the graphs $\G_{comptag(\C)},\ \G_{region(\C)}$ are identical under the renaming
of vertex $i$ of $\G_{comptag(\C)}$ as vertex $i'$ of $\G_{region(\C)}.$ 
\end{lemma}
\begin{proof}
1. By definition, a row of $A(\G_{comptag(\C)})$ corresponds to the vertex set
of a connected component of ${tag(\C)}.$ By Lemma \ref{lem:tagpathregionpath}, we know that
whenever a vertex, say $v_{+}(\delta),$ is path connected in ${tag(\C)}$
to a vertex, say $v_{-}(\delta'),$
it is also path connected in $\T\setminus \C$
and therefore these vertices belong to the same 
connected region of $\T\setminus \C.$ Each vertex set of a connected component
of $tag(\C)$ corresponds to a single connected region in which it lies
in $\T\setminus \C.$
So to every vertex $i$ of $\G_{comptag(\C)},$ there corresponds a unique vertex 
$i'$ of $\G_{region(\C)},$ such that
whenever an  edge $\delta$ leaves (enters) a vertex $i$
in $\G_{comptag(\C)},$
it also leaves (enters) the
corresponding vertex $i'$ of   $\G_{region(\C)}.$
This is also true of selfloops of $\G_{comptag(\C)}.$ 
Thus, if selfloop $\delta$ is incident at $i$ in $\G_{comptag(\C)},$ it is incident
at $i'$ in $\G_{region(\C)}.$ 
\\

2. Since, by Lemma \ref{lem:cycleregion}, $\G_{region(\C)}$ is connected, it has no isolated nodes. 
Therefore , since the edge sets of both graphs are identical, the map $i\rightarrow i',$ from vertex set of $\G_{comptag(\C)},$ 
to vertex set of $\G_{region(\C)}$
is onto.

If $\C$ is connected it is easy to see that $\G_{comptag(\C)}$
is connected. For connected graphs, by Lemma \ref{lem:incidencerank},
the row rank of the  incidence matrix is one less than the  number of nodes of the
graph.
By Lemma \ref{lem:incidence}, Lemma \ref{lem:cycleregion}, the row space of $\G_{comptag(\C)}$  is contained in the
row space of $\G_{region(\C)}.$
Thus the number of nodes of $\G_{comptag(\C)}$
is less or equal to the number of nodes of $\G_{region(\C)}.$
But the map $i\rightarrow i'$ is onto and therefore we must have the number of nodes 
the same for both graphs and therefore the map $i\rightarrow i'$ is one to one onto. 
Further, an edge $\delta$ is directed from  $i$ to $j$ in $\G_{comptag(\C)}$
iff it is directed from $i'$ to $j'$ in $\G_{region(\C)}.$
Thus, under the renaming of node $i$ of $\G_{comptag(\C)}$
as node $i'$ of $\G_{region(\C)},$
the two graphs are identical.

\end{proof}
\begin{remark}
If the $2-$complex $\C$ is not connected, we first build $\G_{comptag(\C_i)}$
for each of the connected components $\C_i$ of $\C.$ 
It will turn out that each row of  $A(\G_{comptag(\C_i)})$ except one 
would be identical to a corresponding row of $A(\G_{region(\C)}).$
If we add these exceptional rows we would get the row corresponding to the 
external node of $\G_{region(\C)}$
and this would complete the construction of $A(\G_{region(\C)})$ and therefore of $\G_{region(\C)}.$
\end{remark}
\section{Matroid dual of the   $(2,1)$ skeleton of $\C$}
\label{sec:matroid}
In this section, we explore the matroidal relation between the $2-$complex $\C$
and the graph $\G_{region(\C)}.$
Our statements are in general correct for $2-$complex $\C$ and $\G_{region(\C)}.$
If $\C$ is connected $\G_{comptag(\C)}=\G_{region(\C)}.$
Otherwise, we can build $\G_{region(\C)}$ from the disconnected $\G_{comptag(\C)}$ 
by fusing the nodes corresponding to the external region in each connected component of the latter.

It is well known that if a graph is nonplanar, its matroid dual will not correspond to a graph.
It is a natural question to ask whether the dual of a nonplanar graph has a simple
representation. We will sketch arguments to show that any nonplanar 
graph  is the minor of another nonplanar graph whose matroid dual is the  $(2,1)$ skeleton
of a $2-$complex embedded in $\mathbb{R}^3.$

We first discuss some preliminary ideas from matroid theory.

The following two results are fundamental in matroid theory \cite{tutte} , \cite{welsh76}, \cite{oxleybook}.
\begin{theorem}
Let $\M=(S,\mathcal{I})$ be a matroid
 and let $T\subseteq S.$
\begin{enumerate}
\item Let  ${\mathcal{I}}\circ T$ be the family of members of ${\mathcal{I}}$
contained in $T.$ Then $\M\circ T:= (T, {\mathcal{I}}\circ T)$ is a matroid.

\item Let  ${\mathcal{I}}\times T$ be the family of all members $X$ of ${\mathcal{I}}$
contained in $T$ with the property that $X\cup \B_{S-T}$ is a member of ${\mathcal{I}}$
whenever $\B_{S-T}$ is a base of $\M\circ (S-T).$
Then  $\M\times T:= (T, {\mathcal{I}}\times T)$ is a matroid.

\item
 Let ${\mathcal{I}}^*$ be the family of subsets of $S$
which are subsets of cobases of $\M.$
Then $(S, {\mathcal{I}}^*)$ is a matroid.
\end{enumerate}
\end{theorem}
\begin{definition}
\label{def:matroiddotcross}
Let $\M$ be a matroid on $S$ and let $T\subseteq S.$
The matroid $\M\circ T:= (T, {\mathcal{I}}\circ T)$ is   called the restriction of $\M$ to $T.$
The matroid $\M\times T:= (T, {\mathcal{I}}\times T)$ is  called the contraction of $\M$ to $T.$
The matroid $\M^*:= (S, {\mathcal{I}}^*)$ is  called the dual of $\M.$

\end{definition}

\begin{theorem}
\label{thm:matroidspace}
Let $\V$ be a vector space on a finite set $S$ and let $T\subseteq S.$  Then
\begin{enumerate}
\item $\M(\V^{\perp})=(\M(\V))^*.$
\item $\M(\V\circ T)=(\M(\V))\circ T.$
\item $\M(\V\times T)=(\M(\V))\times T.$
\end{enumerate}

\end{theorem}



We have, by Lemma \ref{lem:cycleregion},  that the $2-$coboundary space of the $2-$complex $\C$ and the $1-$coboundary space 
of the graph $\G_{region(\C)}$ are complementary orthogonal.
The matroid $\M(\V(\G))$ associated with the $1-$coboundary space $\V(\G)$
of $\G$ is said to be also associated with $\G$ and denoted $\M(\G).$
Let us denote by $\V(\C),$  the $2-$coboundary space of $\C$
and call $\M(\V(\C)),$
the matroid $\M(\C)$
associated with the  $(2,1)$ skeleton of $\C.$
By Theorem \ref{thm:matroidspace}, we have that $\M(\V^{\perp})=(\M(\V))^*.$
So $(\M(\C))^*= \M(\G_{comptag(\C)}).$
Thus the matroid dual of the  $(2,1)$ skeleton of $\C$ is a graph which may in general 
be nonplanar.

The natural question is the converse. Can one build a $2-$complex $\C$
whose  $(2,1)$ skeleton  is the matroid dual of a nonplanar graph $\G_{orig}$ ?
We will sketch an algorithm for building a larger graph $\G_{large}$
which has such a matroid dual and from which by 
deleting a single vertex and edges incident on it,
we can get back $\G_{orig}.$
It is well known that a graph $\G_{orig}$ is nonplanar iff it 
contains  the Kuratowski graphs $K_{5},K_{3,3}$ as minors.
Therefore if $\G_{orig}$ is nonplanar, so is $\G_{large}.$

The method we use is to construct  regions surrounding the individual  nodes 
in such a way that regions corresponding to two nodes have a common boundary iff
there is an edge in $\G_{orig}$ between them.
But this method leaves us with an extra `external' region which has common boundary 
with every one of the regions corresponding to the nodes of $\G_{orig}.$
The $2-$complex made up of the boundary triangles of all the regions
will turn out to be the matroid dual of a larger graph $\G_{large}.$

Figure \ref{fig:planar2} indicates a simplified version of the problem.
The original graph $\G_{orig}$ should be thought of as embedded  in $\mathbb{R}^3$
even though for diagrammatic simplicity we have used a planar representation.
We surround the nodes by `tubes' $\Reg_1, \cdots , \Reg_6,$ whose surfaces are triangulated but not shown.
The tubes  have certain common triangles with other tubes. These are indicated as $e_1',\cdots , e_6'.$ 
In addition, there is an external region $\Reg_{ext}$ which has common triangles 
with all the tubes. Let the complex $\C$ be composed of all surface triangles,
their boundary edges and vertices. The resulting region graph, $\G_{large}$
of $\C,$ 
is shown alongside. The dotted lines indicate sets of parallel edges corresponding
to triangles common to the external region and the other `tubes'.
The $2-$coboundary space $\V(\C)$ and the $2-$cycle space $\V(\C)^{\perp}$ 
respectively of $\C$ would be the $1-$cycle space and $1-$coboundary space  of 
$\G_{large}$ if the orientations of triangles $\C$ and the edges of $\G_{large}$ 
are consistent.
\begin{figure}
\label{fig:planar2}
\centering
 \includegraphics[width=5in]{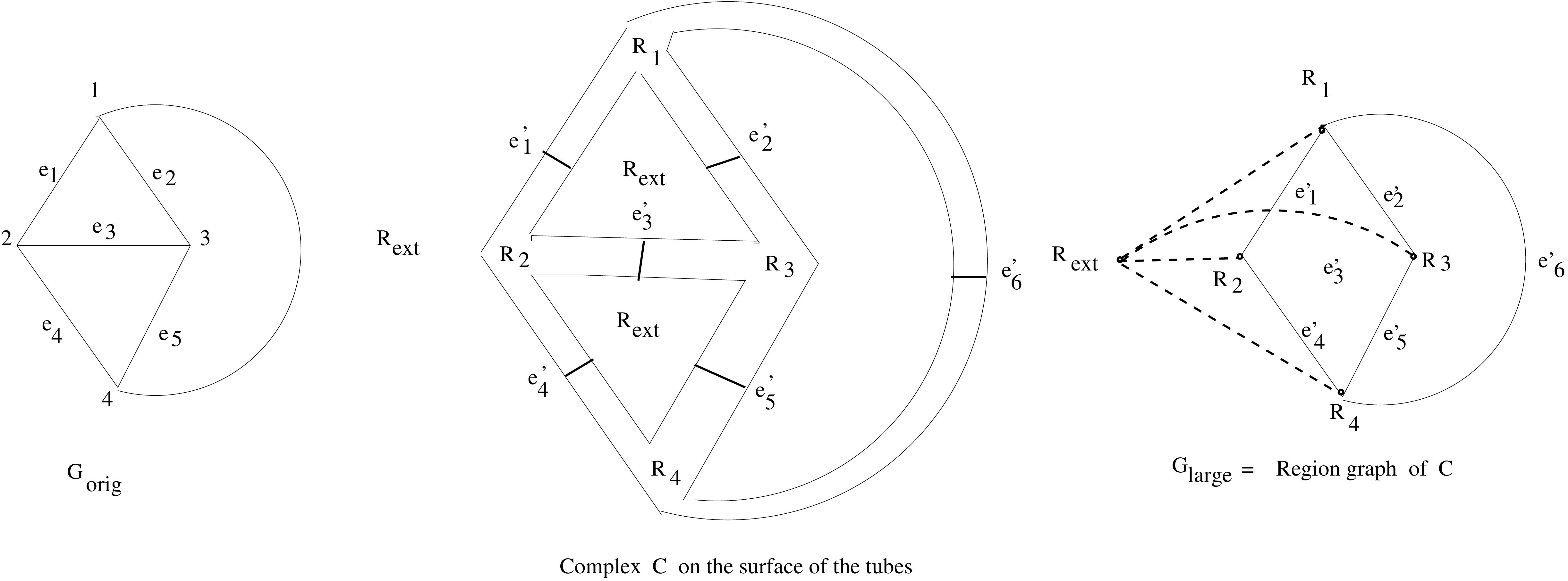}
 \caption{From $\G_{orig}$ to $\G_{large}$}
 \label{fig:triangles_edge}
\end{figure}

\begin{remark}
The matroid dual has to be distinguished from the Lefchetz/Poincar\'{e} duality of complexes (see \cite{hatcher}, \cite{munkres}).
In general the latter do not lead to matroid duality. For instance, when a graph is drawn on a torus without edges crossing each other, the coboundary space of its Lefschetz dual will have dimension less (by genus of torus)
than the dimension of the cycle space of the original graph. In general when a $2-$complex $\C$ is embedded 
in a three dimensional space, say $\mathcal{S},$  the region graph will be the Lefschetz dual but unless the 
space $\mathcal{S}$ is contractible, the region graph may not be the matroid dual of $\C$ and will not allow us to write equations for an electrical $2-$network
based on $\C.$
\end{remark}

\section{Dual of an electrical $2-$network}
\label{sec:dualnetwork}
From the discussions in the previous sections, it is reasonable to expect that if an electrical $2-$network is built
on a $2-$complex $\C,$ we could build a `dual' electrical network of the graph based kind, the solution of which
would yield a solution of the  $2-$network.
We spell out the details in this section.

First we dualize constraints $K_1,K_2$ of Subsection \ref{def:flux_mmf}.\\

Constraint $K_1$ (Generalized KVL) can be simplified to read\\
``the net outward flux $\phi_S$ through $2-$cells of $\C$ which are boundaries of connected regions
of tetrahedra of $\C_{\T}$ is zero.
"\\
Constraint $Dual \ K_1$ (KCL) should read\\ 
``the net current $\phi_S$ away from a node of $\G_{region(\C)}$ 
is zero."
\\

Constraint $K_2$ (Generalized KCL) can be simplified to read\\
``the current adjusted
mmf vector $m_S$ 
is a $2-$cycle of $\C.$"

By Lemma \ref{lem:cycleregion} we know that $2-$cycles of $\C$ and $1-$coboundaries of $\G_{region(\C)}$
are identical. So\\
Constraint $Dual \ K_2$ (KVL) should read

``the vector $m_S$ is a $1-$coboundary of $\G_{region(\C)}.$"
\\

We defined an  electrical $2-$network $\N$ as a pair $(\C_S,\D_S),$
where\\
$\C_S$ is a $2-$complex embedded in $\mathbb{R}^3$ with $S$ as its $2-$cells,\\
$\D_S, $ the device characteristic, is a collection of ordered pairs $(\phi_S,m_S)$ of vectors on $S$ over $\mathbb{R}.$

Therefore, we may define the dual network $\N^d$ to $\N$ as the pair $(\G_{region(\C)}, \D_S).$\\
A solution of $\N^d$ 
is an ordered pair $(\phi_S,m_S)$
satisfying the following constraints
\begin{enumerate}
\item $\phi_S $ satisfies Constraint $Dual \ K_1.$
\item $m_S $ satisfies Constraint $Dual \ K_2.$
\item  $(\phi_S,m_S)\in \D_S.$
\end{enumerate}

Observe that $\phi_S $ satisfies Constraint $K_1$ iff it satisfies Constraint $Dual \ K_1$
and 
$m_S $ satisfies \\
Constraint $K_2$ iff it satisfies Constraint $Dual \ K_2.$
Thus $(\phi_S,m_S)$ is a solution of  $\N$ iff it is a solution of  $\N^d.$ 

Subsection \ref{subsec:cycleeqns} contains a description of cycle based
solution of $\N.$ This corresponds to coboundary based solution of $\N^d.$
In particular, when the device characteristic permits $\phi_S$ to be expressed as an affine function  of 
 $m_S$, one could use nodal analysis based on node potentials for $\N^d.$

It is of interest to know how results of the `standard' electrical networks
carry over to electrical $2-$networks.

We have already seen that Tellegen's Theorem carries through (Theorem \ref{gtellegen}). Ideas of topological network theory such as multiport decomposition and topological transformation
\cite{book}
go through essentially unchanged because we are able to build the dual nonplanar graph to the complex
and therefore the dual electrical network.

Kirchhoff's tree theorem \cite{seshu} states that the equivalent conductance seen across a pair of terminals $a,b$
of a network composed of resistors equals 
$$ \frac{\mbox{sum of all tree conductance products}}{{\mbox{sum of all 2-tree conductance products separating} \ a,b}}.$$
Here a 2-tree separating $a,b$ is a loop free set obtained from a tree by removing an edge so that
$a,b$ occur in different connected components of the tree and
a tree (2-tree) conductance product is the product of the conductances in a tree (2-tree) of the resistive network.
This theorem clearly 
has an analogue for the $2-$network $\N.$ 
Let exactly one of the triangles, say $\delta ,$  carry a current $i'$ around it consistent with the orientation.
After solving $\N^{d},$ we can obtain $f_{\delta}:= \phi_{\delta}= g i'.$ 
To compute $g$ by Kirchhoff's tree formula, note that the triangle $\delta ,$ and the current $i'$ translate 
in the dual network to an edge $e_{\delta}$ composed of a resistor (of resistance value $=$ reluctance of triangle) in series with a  voltage
source of value $E:=i'$ (see Figure \ref{fig:primaldual}). 
\begin{figure}
 \label{fig:primaldual}
\centering
 \includegraphics[width=7in]{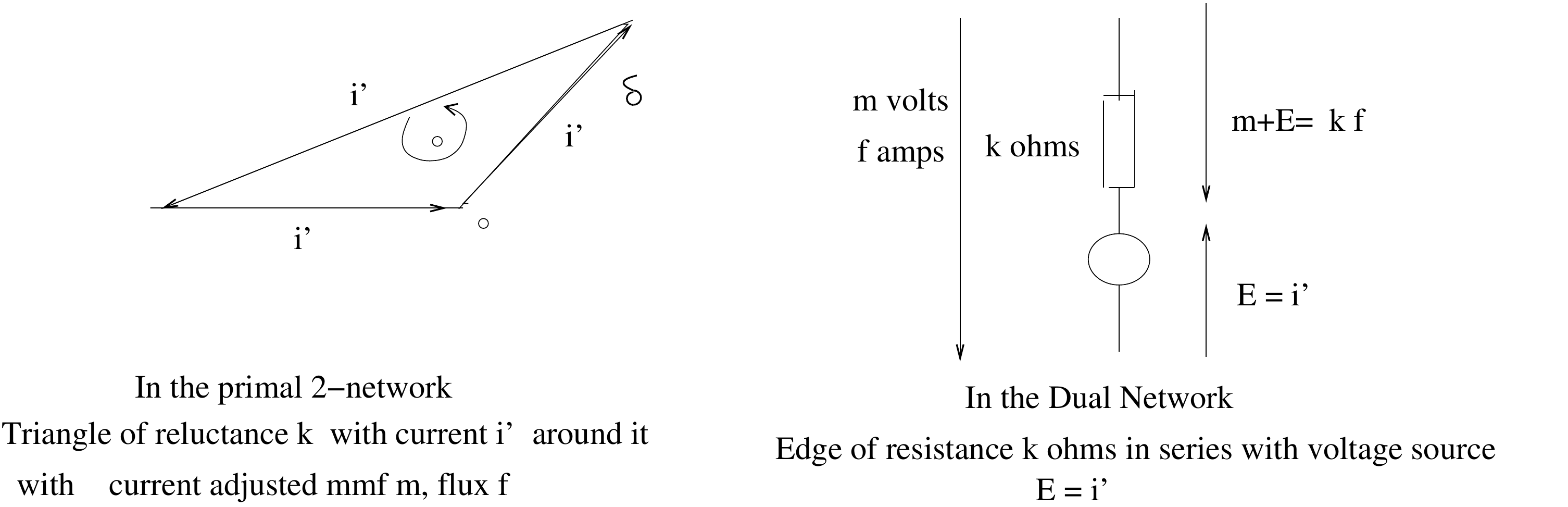}
 \caption{Triangle in a $2-$ network and the corresponding edge in the dual network}
\end{figure}

In the dual network, after solution, we have the relation, current $f := f_{\delta} = g E.$
The quantity $g$ is the equivalent conductance as seen by the source $E$ when no other source is present.   
We can write an expression for the inverse reluctance $g$ using Kirchhoff's tree formula for the dual network.
This can be translated to an expression involving the cobases of $\M(\C)$ and subcobases
which become a cobase when  $\delta$ is added.
Let us call the latter `$\delta-$ friendly'  subcobases.
Then the equivalent  inverse reluctance seen by the current loop $i'$,
would be
$$ \frac{\mbox{sum of all cobase }
(reluctance)^{-1} \mbox{products}}{\mbox{sum of all}\   \delta- \mbox{friendly  subcobase }
 (reluctance)^{-1} \mbox{products}}
.$$

\section{Generalizations}
\label{sec:gen}
%
We discuss in this section how the ideas of this paper go through if we have an $(n-1)-$complex $\C$ embedded in
$\mathbb{R}^n.$


As before, we will assume that there is a large
$n-$simplex $\T,$ whose interior contains $\C.$
The $n-$simplex
$\T$ is decomposed into a set of $n-$simplices  whose interiors do not intersect and which together with their faces
constitute the $n-$complex $\C_{\T} .$
We assume $\C$ to be a subcomplex of $\C_{\T} .$
Next, because $\T$ is contractible we have,
by Theorem \ref{Hurewicz}, that for $0<j\leq n,$ every $(j-1)-$cycle
of the complex is the boundary of a $j-$chain.
When $n>3,$ we have to consider two cases, $j=n$ and $j<n.$

The ideas of this paper go through essentially unchanged for the $j=n$ case.

We will then have  the $(n-1)-$coboundary space of $\C$ as regular.
The electrical $(n-1)-$circuit would be defined essentially identically,
with the constraints $K_1,K_2$ modified by replacing $2-$cycle and $2-$coboundary spaces
by $(n-1)-$cycle and $(n-1)-$coboundary spaces
respectively. The generalized Tellegen's Theorem has an identical proof replacing $2$ by $(n-1)$ at appropriate
places. The $(n-1)-$dimensional analogue of the notion of triangles coming together at an edge requires some careful handling.
But finally we will have a matroid dual which is the region graph.
This region graph can again be constructed by an analogue of the sliding algorithm.
\subsection{Sketch of the $(n-1)-$complex solution}
\label{subsec:sketch}

The analogue of an edge is an $(n-2)-$dimensional simplex.
The collection of triangles coming together at an edge corresponds to
a collection $\Sigma$ of $(n-1)-$dimensional simplices.
We describe a general way of reducing this problem to line segments 
coming together at a node which will also work in the case considered in this paper.

Consider an $(n-2)-$dimensional simplex $e$ contained in $\mathbb{R}^n.$
 Suppose the barycentre of $e$ is the origin. Let $\V$ be
the two dimensional vector space that 
is the orthogonal complement of the linear span of  $e.$ Let
$e$ be a face of each member of a collection  $\Sigma$ of $(n-1)-$dimensional simplices $\delta_1, \cdots , \delta_k.$ 
Consider $\V$ intersected with the union of all the
simplices in $\Sigma$. This set is the union of line segments $l_{\delta_i}$ having the
origin (representing $e$) as one endpoint, as can be seen from dimensional considerations.
Let all the simplices $\delta_i$ be oriented so that $e$ appears
with a positive sign in $\partial(\delta_i).$
We could take our convention as clockwise rotation around the origin  for fixing 
the order in which the $l_{\delta_i}$  are encountered. 
The vertices $v_-(\l_{\delta_i}), v_+(\l_{\delta_i})$ are 
to either side of $l_{\delta_i}$ so that as we move clockwise about the origin, 
we encounter $v_-(\l_{\delta_i})$ before crossing $l_{\delta_i}$
 to encounter $v_+(\l_{\delta_i}).$

The construction of $tag(\C), \G_{comptag(\C)}, \G_{region(\C)}$ is as in the $2-$complex case
and it will turn out that $\G_{comptag(\C)}= \G_{region(\C)}$ when $\C$ is connected
and that the $(n-1)-$cycle space of $\C$ is the row space of the incidence matrix of
$\G_{comptag(\C)}.$

So in this case of an $(n-1)-$complex $\C$ embedded in
$\mathbb{R}^n,$ the electrical $(n-1)-$network problem can be solved as a graph  based electrical network problem.

\subsection{$j<n$ case}
In this case $\C$ is a $(j-1)-$complex.
Constraint $K_1$ would read ``the flux vector $\phi_S$ is orthogonal to $(j-1)-$cycles of $\C$ which are boundaries of $j-$chains of $\C_{\T}.$" Because of contractibility of $\T$ in $\mathbb{R}^n,$ 
$(j-1)-$cycles of $\C$  would reduce to boundaries of $j-$chains of $\C_{\T}.$
However,
these $j-$chains cannot be regarded as regions of $\C_{\T}.$ 
There would still be a non trivial generalized Tellegen's Theorem stating that the space of $\phi_S$ 
and the space of $m_S$ are complementary orthogonal.
Complete unimodularity would not hold for the $(j-1)-$ coboundary or cycle space.  
Nor can we use ideas like region graph. The computation of a maximal linearly independent set of rows
of the $(j-1)-$coboundary matrix will involve linear algebraic computations (see Subsection \ref{subsec:solution}).
\section{Conclusion}
\label{sec:conclusion}
In this paper we discussed a generalization of Kirchhoff's laws on graphs, to the case of $2-$complexes.
We were led to this generalization while attempting to solve a physical problem 
involving fluxes and mmfs in Euclidean space. We showed that there was a non trivial generalization
of Tellegen's Theorem on the complementary orthogonality of voltage and current spaces,
 to that of flux and mmf spaces in the new context. This helped 
us to define an electrical $2-$network on a $2-$complex $\C$ and discuss procedures for its solution. 
We gave linear time algorithms for building auxiliary graphs
 $tag(\C)$ and $\G_{comptag(\C)}.$ When $\C$ was connected, we showed that $\G_{comptag(\C)}$ was the matroid dual of $\C,$ i.e., that the column matroid on the incidence matrix of $\G_{comptag(\C)}$ was dual to the column matroid of 
the $2-$coboundary matrix of $\C.$ Using this duality, we showed how to build in linear time, 
a graph based dual network to the 
 given $2-$network, whose solutions were identical to solutions of the latter after appropriate renaming of variables.
We inferred that preprocessing for the new
class of networks for  parallelization,  
using methods such as multiport decomposition and topological transformation, 
was as easy as it was for graph based networks.
Finally, we discussed generalizations of networks on $2-$complexes embedded in
$\mathbb{R}^3$, first to  $(n-1)-$complexes and then to $j-$complexes, $0<j<n-1,$ embedded in $\mathbb{R}^n.$ 

\section*{Acknowledgements}
\thispagestyle{empty}

The authors  would like to acknowledge helpful discussions with Arvind Nair.


\bibliographystyle{elsarticle1-num}

\bibliography{references}

\end{document}